\documentclass[11pt]{amsart}
\usepackage{amssymb,amsmath,amssymb}
\usepackage{abstract}
\usepackage{amsthm}
\usepackage{mathrsfs}
\usepackage{graphicx}
\usepackage{subfigure}
\usepackage{caption}
\usepackage{float}
\usepackage{epstopdf}
\usepackage{makecell,rotating,multirow,diagbox}
\usepackage{booktabs}
\usepackage{makecell}
\usepackage{color}
\usepackage{url}
\usepackage[colorlinks,linkcolor=blue,anchorcolor=blue,citecolor=red]{hyperref}
\usepackage{appendix}

\def\epsilon{\varepsilon}
\def\m{\textbf{m}}
\def\e{\textbf{e}}
\def\n{\textbf{n}}

\def\bu{\textbf{u}}
\def\d{\,\mathrm{d}}

\def\bx{\boldsymbol{x}}
\def\by{\boldsymbol{y}}
\def\bz{\boldsymbol{z}}

\def\bdiv{\mathrm{div}}
\def\hd{\mathbf{h}_{\mathrm{d}}}
\def\ha{\mathbf{h}_{\mathrm{a}}}
\def\G{\mathcal{G}_{\mathcal{L}}}
\newcommand{\abs}[1]{\lvert#1\rvert}

\newcommand{\nn}{\nonumber}

\newtheorem{theorem}{Theorem}
\newtheorem{assumption}{Assumption}

\newtheorem{lemma}{Lemma}
\newtheorem{remark}{Remark}[section]

\begin{document}	
\title[Homogenization of the LLG equation]{Homogenization of the Landau-Lifshitz-Gilbert equation with natural boundary condition}
	
\author{Jingrun Chen}
\address{School of Mathematical Sciences, University of Science and Technology of China, Hefei, Anhui 230026, China; Suzhou Institute for Advanced Research, University of Science and Technology of China, Suzhou, Jiangsu 215123, China}
\email{jingrunchen@ustc.edu.cn}

\author{Jian-Guo Liu}
\address{Department of Mathematics and Department of Physics, Duke University, Box 90320, Durham NC 27708, USA}
\email{jliu@phy.duke.edu }

\author{Zhiwei Sun}
\address{School of Mathematical Sciences, Soochow University, Suzhou, Jiangsu 215006, China}
\email{20194007008@stu.suda.edu.cn}

\date{\today}

\maketitle

\begin{abstract}
The full Landau-Lifshitz-Gilbert equation with periodic material coefficients and natural boundary condition is employed to model the magnetization dynamics in composite ferromagnets. In this work, we establish the convergence between the homogenized solution and the original solution via a Lax equivalence theorem kind of argument. There are a few technical difficulties, including: 1) it is proven the classic choice of corrector to homogenization cannot provide the convergence result in the $H^1$ norm; 2) a boundary layer is induced due to the natural boundary condition; 3) the presence of stray field give rise to a multiscale potential problem. To keep the convergence rates near the boundary, we introduce the Neumann corrector with a high-order modification. Estimates on singular integral for disturbed functions and boundary layer are deduced, to conduct consistency analysis of stray field. Furthermore, inspired by length conservation of magnetization, we choose proper correctors in specific geometric space. These, together with a uniform $W^{1,6}$ estimate on original solution, provide the convergence rates in the $H^1$ sense.
\end{abstract}

\section{Introduction}\label{sec:introduction}
The intrinsic magnetic order of a rigid single-crystal ferromagnet over a region $\Omega\subset \mathbb{R}^n, n=1,2, 3$ is described by the magnetization $\mathbf{M}$ satisfying
\begin{equation*}
\mathbf{M} = M_s(T)\m, \quad \mbox{a.e. in $\Omega$},
\end{equation*}
where the saturation magnetization $M_s$ depends on the material and the temperature $T$. Below Curie temperature, $M_s$ is modeled as a constant.

A stable structure of a ferromagnet is mathematically characterized as the local minimizers of the Landau-Lifshitz energy functional \cite{LandauLifshitz1935}
\begin{equation*}
\begin{aligned}
\G[\m] 
:=& \int_{\Omega} a(\bx) \vert \nabla \m\vert^2 \d \bx
+ 
\int_{\Omega} K(\bx) \left( \m\cdot \bu \right)^2  (\m) \d \bx\\
&-
\mu_0\int_{\Omega} \hd[M_s\m]\cdot M_s\m \d \bx
-
\int_{\Omega} \ha\cdot M_s\m \d \bx\\
=:&
\mathcal{E}(\m)+\mathcal{A}(\m)+\mathcal{W}(\m)+\mathcal{Z}(\m).
\end{aligned}
\end{equation*}
$\mathcal{E}(\m)$ is the exchange energy, which penalizes the spatial variation of $\m$. The matrix $a = (a_{ij})_{1\le i,j \le 3}$ is symmetric, uniformly coercive and bounded, i.e.,
\begin{equation}\label{uniformly coercive}
\left\{ \begin{aligned}
& \sum_{i,j=1}^n a_{ij}(\bx) \eta_i\eta_j \ge a_{\mathrm{min}} \abs{\boldsymbol{\eta}}^2 \quad
\mbox{for any $\bx\in\mathbb{R}^n$, $\boldsymbol{\eta}\in \mathbb{R}^n$},\\
& \sum_{i,j=1}^n a_{ij}(\bx) \eta_i\xi_j \le a_{\mathrm{max}} \abs{\boldsymbol{\eta}}\abs{\boldsymbol{\xi}} \quad \mbox{for any $\bx\in\mathbb{R}^n$, $\boldsymbol{\eta},\ \xi\in \mathbb{R}^n$}.
\end{aligned} \right.
\end{equation}
In the anisotropy energy $\mathcal{A}(\m)$, $\bu$ is the easy-axis direction which depends on the crystallographic structure of the material. The anisotropy energy density is assumed to be a non-negatively even and globally Lipschitz continuous function that vanishes only on a finite set of unit
vectors (the easy axis). The third term $\mathcal{W}(\m)$ is the magnetostatic self-energy due to the dipolar magnetic field, also known as the stray field $\hd[\m]$. For an open bounded domain $\Omega$ with a Lipschitz boundary, the magnetization $\m \in L^p(\Omega, \mathbb{R}^3)$ generates a stray field satisfying
\begin{equation}\label{define hd}
\hd[\m] =\nabla U_m,
\end{equation}
where the potential
$U_m$ solves
\begin{equation}\label{eqn:stray field eqn}
\Delta U_m = -\bdiv(\m \mathcal{X}_{\Omega}),\quad \mbox{in $D'(\mathbb{R}^3)$}
\end{equation}
with $\m \mathcal{X}_{\Omega}$ the extension of $\m$ to $\mathbb{R}^3$ that vanishes outside $\Omega$. The existence and uniqueness of $U_m$ follows from the Lax-Milgram Theorem and $U_m$ satisfies the estimate \cite{Praetorius2004AnalysisOT}
\begin{equation}\label{bound of stray field}
\Vert \hd[\m] \Vert_{L^p(\Omega)}\le\Vert \m \Vert_{L^p(\Omega)}  \quad 1<p<\infty.
\end{equation}
The last term $\mathcal{Z}(\m)$ is the Zeeman energy that models the interaction between $\m$ and the externally applied magnetic field $\ha$.

For a composite ferromagnet with periodic micorstructures, the material constants are modeled with periodic material coefficients with period $\epsilon$, i.e., $a^\epsilon= a(\bx/\epsilon)$, $K^\epsilon = K(\bx/\epsilon)$, $M^\epsilon = M_s(\bx/\epsilon)$, with functions $a$, $K$, $M_s$ periodic over $Y=[0,1]^n$. The associated energy reads as
\begin{equation}\label{Gibbs landau energy of epsilon}
\begin{aligned}
\G^\epsilon[\m] :=& \int_{\Omega} a^{\epsilon}(\bx) \vert \nabla \m\vert^2 \d \bx
+ 
\int_{\Omega} K^\epsilon(\bx) \left( \m\cdot \bu \right)^2 \d \bx\\
&- \mu_0
\int_{\Omega} \hd[M^\epsilon\m]\cdot M^\epsilon\m \d \bx
-
\int_{\Omega} \ha\cdot M^\epsilon\m \d \bx.
\end{aligned}
\end{equation}
It is proved in \cite{Alouges_2015} that $\G^\epsilon[\m]$ is equi-mild coercive in the metric space $(H^1(\Omega, S^2), d_{L^2(\Omega,S^2)})$ and $\Gamma$-converges to the functional $\mathcal{G}_{\mathrm{hom}}$ defined as
\begin{align}\nn
\mathcal{G}_{\mathrm{hom}}[\m] 
=& \int_{\Omega} a^0 \vert \nabla \m\vert^2 \d \bx
+ 
\int_{\Omega} K^0 \left( \m\cdot \bu \right)^2 \d \bx
-
\mu_0
(M^0)^2\int_{\Omega} \hd [\m]\cdot \m \d \bx\\\label{homogenized energe}
&-
\mu_0 
\int_{\Omega\times Y}  \big\vert \m \cdot \mathbf{H}_{\mathrm{d}}[M_s(\by)](\by) \big\vert^2 \d \bx \d \by
-
M^0\int_{\Omega} \ha\cdot \m \d \bx,
\end{align}
where $a^0$ is the homogenized tensor
\begin{equation*}
a^0_{ij} = \int_Y\left( a_{ij} + \sum_{k= 1}^n a_{ik}\frac{\partial \chi_j}{\partial y_k} \right) \d \by,
\end{equation*}
the constants $M^0$ and $K^0$ are calculated by
\begin{equation*}
M^0 = \int_Y M_s(\boldsymbol{y}) \d \boldsymbol{y},
\quad 
K^0 =  \int_Y K(\boldsymbol{y}) \d \boldsymbol{y},
\end{equation*}
and the symmetric matrix-valued function $\mathbf{H}_{\mathrm{d}}[M_s(\by)](\by)= \nabla_{\by} \mathbf{U}(\by)$ with potential function given by
\begin{equation}\label{define H_d}
	\begin{gathered}
		\int_{Y} M_s(\boldsymbol{y})\nabla_{\boldsymbol{y}} \varphi(\boldsymbol{y}) \d \boldsymbol{y}
		=
		- \int_{Y} \nabla_{\boldsymbol{y}} \mathbf{U}(\by) \cdot \nabla_{\boldsymbol{y}} \varphi(\boldsymbol{y}) \d \boldsymbol{y},\\
		\mbox{$\mathbf{U}(\by)$ is $Y$-periodic}, \quad
		\int_{Y} \mathbf{U}(\by)  \d \boldsymbol{y}  = 0,
	\end{gathered}
\end{equation}
for any periodic function $\varphi\in H^1_{\mathrm{per}}(Y)$.

In the current work, we are interested in the convergence of the dynamic problem driven by the Landau-Lifshitz energy \eqref{Gibbs landau energy of epsilon} to the dynamics problem driven by the homogenized energy \eqref{homogenized energe} as $\epsilon$ goes to $0$. It is well known that the time evolution of the magnetization over $\Omega_T=\Omega\times[0, T]$ follows the Landau-Lifshitz-Gilbert (LLG) equation \cite{LandauLifshitz1935,Gilbert1955}
\begin{equation}\label{eqn:LLG system}
\left\{ \begin{aligned}
&\partial_t\m^\epsilon - \alpha \m^\epsilon \times \partial_t\m^\epsilon = 
- (1+\alpha^2) \m^\epsilon \times \mathcal{H}^\epsilon_e(\m^\epsilon) \quad \mbox{a.e. in $\Omega_T$},\\
& \boldsymbol{\nu}\cdot a^\epsilon\nabla \m^\epsilon = 0,\quad \mbox{a.e. on $\partial\Omega\times[0, T]$},\\
& \m^\epsilon(0,x) = \m_{\mathrm{init}}^\epsilon(x),\quad 
\abs{\m_{\mathrm{init}}^\epsilon(x)} = 1\quad \mbox{a.e. in $\Omega$},
\end{aligned} \right.
\end{equation}
where $\alpha>0$ is the damping constant, and the effective field $\mathcal{H}^\epsilon_e(\m^\epsilon)=
- \frac{\delta \G^\epsilon}{\delta \m^\epsilon}
$ associated to the Landau-Lifshitz energy \eqref{Gibbs landau energy of epsilon} is given by
\begin{equation}\label{eqn:H^epslon}
\begin{aligned}
\mathcal{H}^\epsilon_e(\m^\epsilon) 
= 
\mathrm{div}\left( a^\epsilon \nabla \m^\epsilon \right)
-
K^\epsilon \left( \m^\epsilon\cdot \bu \right)\bu 
+ 
\mu_0  M^{\epsilon} \hd[M^{\epsilon}\m^\epsilon]
+
M^{\epsilon} \ha.
\end{aligned}
\end{equation}
Meanwhile, the LLG equation associated to the homogenized energy \eqref{homogenized energe} reads as
\begin{equation}\label{eqn:homogenized LLG system}
\left\{ \begin{aligned}
&\begin{gathered} \partial_t\m_0 - \alpha \m_0 \times \partial_t\m_0 = 
- (1+\alpha^2) \m_0 \times \mathcal{H}^0_e(\m_0)
\end{gathered}\\
& \boldsymbol{\nu}\cdot a^0\nabla \m_0 = 0,\quad \mbox{a.e. on $\partial\Omega\times[0, T]$}\\
& \m_0(0,x) = \m_{\mathrm{init}}^0(x),\quad 
\abs{\m_{\mathrm{init}}(x)} = 1\quad \mbox{a.e. in $\Omega$}
\end{aligned} \right.
\end{equation}
with homogenized effective field $\mathcal{H}^0_e(\m_0)=
- \frac{\delta \mathcal{G}_{\mathrm{hom}} }{\delta \m_0} 
$ calculated by
\begin{equation}\label{homogenized eff field}
\begin{aligned}
\mathcal{H}^0_e(\m_0)
=& 
\mathrm{div}\left( a^0 \nabla \m_0 \right) 
- K^0 \left( \m_0\cdot \bu \right)\bu \\
& + 
\mu_0  (M^0)^2 \hd[\m_0]
+
\mu_0
\mathbf{H}^0_{\mathrm{d}} \cdot \m_0
+
M^0 \ha,
\end{aligned}
\end{equation}
where the matrix $\mathbf{H}^0_{\mathrm{d}} = 
\int_{Y}   M_s(\by) \mathbf{H}_{\mathrm{d}}[M_s(\by)](\by)\d \by$.

Works related to the homogenization of the LLG equation in the literature include \cite{SANTUGINIREPIQUET2007502,choquet2018homogenization,alouges2019stochastic,leitenmaier2021homogenization,leitenmaier2022upscaling,chen2021}. As for the convergence rate, most relevantly, the LLG equation \eqref{eqn:LLG system} with only the exchange term and with the periodic boundary condition is studied in \cite{leitenmaier2021homogenization}. Convergence rates between $\m^\epsilon$ and $\m_0$ in time interval $[0,\epsilon^\sigma T]$ are obtained under the assumption 
\begin{equation}\label{strong assumption}
\Vert \nabla \m^\epsilon \Vert_{L^\infty(\Omega)} \le C, \quad \mbox{for any $t\in [0,\epsilon^\sigma T]$},
\end{equation} 
where $C$ is a constant independent of $\epsilon$ and $\sigma \in [0,2)$. 
As a special case, when $\sigma = 0$ in assumption \eqref{strong assumption}, i.e., $\Vert \nabla \m^\epsilon \Vert_{L^\infty(\Omega)}$ is uniformly bounded over a time interval independent of $\epsilon$, it is proven that $\Vert\m^\epsilon-\m_0\Vert_{L^\infty(0,T;L^2(\Omega))} = \mathcal{O}(\epsilon)$ while 
$\Vert\m^\epsilon-\m_0\Vert_{L^\infty(0,T;H^1(\Omega))}$ is only uniformly bounded without strong convergence rate. 

In this work, we consider the full LLG model \eqref{eqn:LLG system} equipped with the Neumann boundary condition, which is the original model derived by Landau and Lifshitz \cite{LandauLifshitz1935}. We prove the convergence rates between $\m^\epsilon$ and $\m_0$ in the $L^\infty(0,T;H^1(\Omega))$ sense without the strong assumption \eqref{strong assumption}. It is worth mentioning that, the trick to improve the convergence result into $H^1$ sense is to find proper correctors $\m_1$, $\m_2$, such that they satisfy geometric properties
\begin{equation}\label{geo property}
	\m_0\cdot\m_1 = 0,
	\quad\mbox{and}\quad
	\m_0\cdot\m_2 = -\abs{\m_1}^2,
\end{equation}
which are motivated by the length-preserving property of magnetization and asymptotic expansion. 
A familiar definition of classic first-order homogenization corrector $\m_1$ in \eqref{eqn:sln m1} would naturally satisfies first property in \eqref{geo property}; see \cite{leitenmaier2021homogenization}. In this article, the suitable corrector $\m_2$ in \eqref{geo property} is obtained.
By the usage of these properties,
we are able to derive the estimate of consistency error, which is induced by an equivalent form of LLG equation, given in \eqref{eqn:LLG system form 3}, and a sharper estimate than \cite{leitenmaier2021homogenization} in $L^\infty(0,T;H^1(\Omega))$ sense is finally obtained.

Instead of the assumption \eqref{strong assumption}, we prove a weak result that $\Vert\nabla \m^\epsilon\Vert_{L^6(\Omega)}$ is uniformly bounded over a time interval independent of $\epsilon$. Such a uniform estimate is nontrivial for the LLG equation, since the standard energy estimate usually transforms the degenerate (damping) term into the diffusion term and thus the upper bound becomes $\epsilon$-dependent. To overcome this difficulty, we introduce the interpolation inequality when $n\le 3$
\begin{equation}\label{term of high order in intro}
	\begin{aligned}
	\Vert \mathrm{div}\left( a^\epsilon \nabla \m \right) \Vert_{L^{3}(\Omega)}^{3}
	\le &
	C + C\Vert \mathrm{div}\left( a^\epsilon \nabla \m \right) \Vert_{L^{2}(\Omega)}^{6}\\
	& + C
	\Vert \m \times \nabla \left\{ \mathrm{div}\left( a^\epsilon \nabla \m \right)\right\} \Vert_{L^{2}(\Omega)}^{2},
	\end{aligned}
\end{equation}
for the $S^2$-value function $\m$ satisfying homogeneous Neumann boundary condition. 
This inequality can help us
derive a structure-preserving energy estimate, in which the degenerate term is kept in the energy. 

The full LLG model \eqref{eqn:LLG system} we considered contains the stray field, where an independent homogenization problem of potential function in the distribution sense arises, and this complicated the problem when we arrive at the consistency analysis. By using results in \cite{Praetorius2004AnalysisOT} and Green's representation formula, the stray field is rewritten as the derivatives of Newtonian potential. Then we are able to obtain the consistency error by deriving detailed estimate of singular integral for disturbed function and boundary layer.

The effect of boundary layer exists when we apply classic homogenization corrector to the Neumann boundary problem, which would cause the approximation deterioration on the boundary. To avoid this, a Neumann corrector is introduced, which is usually used in elliptic homogenization problems (see \cite{shen2018periodic} for example). 
In this article, we provide a strategy to apply the Neumann corrector to evolutionary LLG equations, by finding a proper higher-order modification.
For a big picture, let us write ahead the linear parabolic equation of error
\begin{equation*}
	\partial_t\e^\epsilon - \mathcal{L}^\epsilon \e^\epsilon + \boldsymbol{f}^\epsilon
	=\boldsymbol{0},
\end{equation*}
whose detailed derivation can be found in \eqref{system of e}.
Following the notation of $\e^\epsilon_{{\mathrm{b}}} = \e^\epsilon - \boldsymbol{\omega}_\mathrm{b}$ with boundary corrector $\boldsymbol{\omega}_\mathrm{b}$, one can find by above equation that an $L^\infty(0,T;H^1(\Omega))$ norm of $\e^\epsilon_{{\mathrm{b}}}$ relies on the boundary data and inhomogeneous term induced by $\boldsymbol{\omega}_\mathrm{b}$, which read as
\begin{equation}\label{terms of boundary corrector}
	\begin{aligned}
		\Vert \boldsymbol{\nu} \cdot a^\epsilon\nabla\{\e^\epsilon + \boldsymbol{\omega}_\mathrm{b}\} \Vert_{B^{-1/2, 2} (\partial\Omega)}
		\quad \mbox{and} \quad 
		\Vert
		\mathcal{L}^\epsilon \boldsymbol{\omega}_\mathrm{b}  \Vert_{ L^{2} (\Omega)}.
	\end{aligned}
\end{equation}
In this end, we divide the corrector $\boldsymbol{\omega}_\mathrm{b}$ into two parts as $\boldsymbol{\omega}_\mathrm{b} = \boldsymbol{\omega}_\mathrm{N} - \boldsymbol{\omega}_\mathrm{M}$, such that they can control two terms in \eqref{terms of boundary corrector} respectively. Here $\boldsymbol{\omega}_\mathrm{N}$ is the Neumann corrector used in elliptic problems (see \cite{shen2018periodic}), and $\boldsymbol{\omega}_\mathrm{M}$ is a modification to be determined. We point out the modification $\boldsymbol{\omega}_\mathrm{M}$ is necessary since calculation implies some bad terms in $\mathcal{L}^\epsilon\boldsymbol{\omega}_\mathrm{N}$ do not converge in $L^{2}$ sense. Therefore we construct following elliptic problem to determine $\boldsymbol{\omega}_\mathrm{M}$:
\begin{equation*}
	\bdiv (a^\epsilon\nabla \boldsymbol{\omega}_\mathrm{M})
	=
	\Big(\mbox{Bad Terms in $\mathcal{L}^\epsilon \boldsymbol{\omega}_\mathrm{N}$}\Big)
\end{equation*}
with proper Neumann boundary condition. Such a solution $\boldsymbol{\omega}_\mathrm{M}$ can be proved to have better estimates than $\boldsymbol{\omega}_\mathrm{N}$, by the observation that all ``Bad Terms in $\mathcal{L}^\epsilon \boldsymbol{\omega}_\mathrm{N}$" can be written in the divergence form. At this point, $\boldsymbol{\omega}_\mathrm{M}$ can be viewed as a high-order modification.

This paper is organized as follows. In the next Section, we introduce the main result of our article and outline the main steps of the proof. In Section \ref{expansion}, multiscale expansions are used to derive the second-order corrector $\m_2$. In Section \ref{consistency}, we deduce that the consistency error $\boldsymbol{f}^\epsilon$ only relies on the consistency error of the stray field, which can be estimated by calculation of singular integral for disturbed function and boundary layer.
In Section \ref{section:boundary corrector}, we introduce the boundary corrector $\boldsymbol{\omega}_\mathrm{b}$, and derive several relevant estimates of it. 
Section \ref{sec:stability} contains the stability analysis in $L^2$ and $H^1$ sense respectively. And we finally give a uniform regularity analysis of $\m^\epsilon$, by deriving a structure-preserving energy estimate in Section \ref{Preliminary}. 

\section{Main result}\label{main result}

To proceed, we make the following assumption
\begin{assumption}\label{assumption}
\item[1.] \textbf{Smoothness} We assume $Y$-periodic functions $a(\by) = (a_{ij}(\by))_{1\le i,j \le 3}$, $K(\by)$, $M_s(\by)$, and the time-independent external field $\ha(\bx)$, alone with boundary $\partial \Omega$, are sufficiently smooth. These together with the definition in \eqref{eqn:sln m1}, \eqref{eqn:system m2 form3}, \eqref{define w_b}, \eqref{define tilde omega} leads to the smoothness of $\m_1$, $\m_2$ and $\boldsymbol{\omega}_\mathrm{b}$.
\item[2.] \textbf{Initial data} Assume $\m_{\mathrm{init}}^0(\bx)$ and $\m_{\mathrm{init}}^\epsilon(\bx)$ are smooth enough and satisfy the Neumann compatibility condition:
\begin{equation}\label{compatibility condition of initial data}
		\boldsymbol{\nu}\cdot a^\epsilon\nabla \m_{\mathrm{init}}^\epsilon(\bx) 
		= 
		\boldsymbol{\nu}\cdot a^0\nabla \m_{\mathrm{init}}^0(\bx)=0 ,  
		\quad \bx\in \partial\Omega.
\end{equation}
Furthermore, we might as well set them satisfying
periodically disturbed elliptic problem:
\begin{equation}\label{initial data}
\begin{aligned}
\bdiv  (a^\epsilon \nabla \m_{\mathrm{init}}^\epsilon(\bx))
=
\bdiv  (a^0 \nabla \m_{\mathrm{init}}^0(\bx)),
\quad \bx\in\Omega.
\end{aligned}
\end{equation}
\eqref{compatibility condition of initial data}-\eqref{initial data} imply $\m_{\mathrm{init}}^0(\bx)$ is the homogenization of $\m_{\mathrm{init}}^\epsilon(\bx)$. Note that the assumption \eqref{initial data} is necessary not only for the convergence analysis in Theorem \ref{thm: Estimates of initial-boundary data}, but also for the uniform estimate of $\m^\epsilon$ in Theorem \ref{regularety 3}.
\end{assumption}
Now let us state our main result:
\begin{theorem}\label{thm:1}
	Let $\m^\epsilon \in L^{\infty}(0,T;H^2(\Omega))$, $\m_0 \in L^{\infty}(0,T;H^6(\Omega))$ be the unique solutions of \eqref{eqn:LLG system} and \eqref{eqn:homogenized LLG system}, respectively. Under Assumption \ref{assumption}, there exists some $T^*\in (0, T]$ independent of $\epsilon$, such that for any $t\in (0,T^*)$ and for 
	$n=2,3$, it holds 
	\begin{equation}\label{eqn:conclution in thm 1}
	\begin{gathered}
	\Vert \m^\epsilon(t) -  \m_0(t) \Vert_{L^2(\Omega)}
	\le \beta(\epsilon), \quad
	\Vert \m^\epsilon(t) -  \m_0(t) \Vert_{H^1(\Omega)}
	\le C \epsilon^{1/2},
	\end{gathered}
	\end{equation}
	where
	\begin{equation}\label{value of beta}
	\beta(\epsilon) =
	\left\{\begin{aligned}
	&C \epsilon [\ln (\epsilon^{-1} + 1)]^2, \quad \mbox{when $n =2$,}\\
	&C \epsilon^{5/6}, \quad \mbox{when $n =3$.}
	\end{aligned}\right.
	\end{equation}
	In the absence of the stray field, i.e., $\mu_0 = 0$, then it holds for any $t\in (0,T^*)$ and for  $n=1,2,3$
	\begin{equation}\label{value of sigma 2}
	\begin{gathered}
	\Vert \m^\epsilon(t) -  \m_0(t) \Vert_{L^2(\Omega)}
	\le C \epsilon [\ln (\epsilon^{-1} + 1)]^2, \\
	\Vert\m^\epsilon(t) - \m_0(t) - (\boldsymbol{\Phi} - \boldsymbol{x} )\nabla\m_0(t) \Vert_{H^1(\Omega)}
	\le C \epsilon [\ln (\epsilon^{-1} + 1)]^2,
	\end{gathered}
	\end{equation}
	where $\bx$ is spatial variable, $\boldsymbol{\Phi}= (\Phi_i)_{1\le i \le n}$ is the corrector defined in \eqref{define Phi}. Constant $C$ depends on the initial data $\m_{\mathrm{init}}^\epsilon$ and $\m_{\mathrm{init}}^0$, but is independent of $\epsilon$.
	\begin{remark}
	Comparing \eqref{eqn:conclution in thm 1} and \eqref{value of sigma 2}, one can see that in the $L^2$ norm, the stray field makes little influence when $n = 2$, but causes $1/6$-order loss of rate when $n=3$. In the $H^1$ norm, however, the stray field leads to $1/2$-order loss of rate in both cases. Such a deterioration of convergence rate is induced since the zero-extension has been applied for stray field \eqref{eqn:stray field eqn}, which introduces a boundary layer.
	\end{remark}
	\begin{remark}
	The logarithmic growth
	$[\ln (\epsilon^{-1} + 1)]^2$ in \eqref{value of sigma 2} is caused by the Neumann corrector $(\Phi - \boldsymbol{x} )\nabla\m_0$. For problems \eqref{eqn:LLG system} and \eqref{eqn:homogenized LLG system} with periodic boundary condition over a cube, by replacing the Neumann corrector in \eqref{value of sigma 2} with the classical two-scale corrector, a similar argument in the current work leads to
		\begin{equation}\label{convergence in periodic condition}
		\begin{gathered}
		\Vert \m^\epsilon - \m_0 - \boldsymbol{\chi} \nabla\m_0  \Vert_{H^1(\Omega)}
		\le C \epsilon,
		\end{gathered}
		\end{equation}
		where $\boldsymbol{\chi} = (\chi_i)_{1\le i \le n}$ is defined in \eqref{eqn:chi_j}.
		
		Note that \eqref{convergence in periodic condition} is consistent with the $L^2$ result in \cite{leitenmaier2021homogenization}. However, only the uniform boundedness in $H^1$ has been shown in \cite{leitenmaier2021homogenization},  while our results \eqref{value of sigma 2} and \eqref{convergence in periodic condition} imply that it maintains the same convergence rate in $L^2$ and $H^1$ norm, by choosing the correctors satisfying specific geometric property \eqref{geo property}.
	\end{remark}
\end{theorem}

\subsection{Some notations and Lax equivalence type theorem}
Recall that a classical solution to \eqref{eqn:LLG system} also satisfies an equivalent form of equation, reads
\begin{equation}\label{eqn:LLG system form 3}
	\mathcal{L}_{\mathrm{LLG}}(\m^\epsilon)
	:=
	\partial_t\m^\epsilon - \alpha \mathcal{H}^\epsilon_e(\m^\epsilon) 
	+ \m^\epsilon \times \mathcal{H}^\epsilon_e(\m^\epsilon) 
	- \alpha  g_l^\epsilon[\m^\epsilon]
	\m^\epsilon
	=0,
\end{equation}
where
the $g_l^\epsilon[\cdot]$ is the energy density calculated by
\begin{equation}\label{energy density of epsilon}
	g_l^\epsilon[\m^\epsilon] =  a^{\epsilon} \vert \nabla \m^\epsilon\vert^2 
	+ 
	K^{\epsilon} \left( \m^\epsilon\cdot \bu \right)\bu
	-
	\hd[M^\epsilon\m^\epsilon]\cdot M^\epsilon\m^\epsilon
	-
	\ha\cdot M^\epsilon\m^\epsilon .
\end{equation}
For convenience, we also define a bilinear operator deduced from \eqref{energy density of epsilon}, which reads
\begin{equation*}
	\mathcal{B}^\epsilon[\m,\n] = a^\epsilon \nabla \m \cdot 
	\nabla \n 
	+ K^\epsilon \left( \m\cdot \bu \right)\left( \n\cdot \bu \right) 
	-\mu_0  \hd [M^{\epsilon}\m]\cdot M^{\epsilon}\n.
\end{equation*}

Now let us set up the equation of error, in terms of Lax equivalence theorem kind of argument. Define the approximate solution 
\begin{equation}\label{def approximate sln}
	\widetilde{\m}^\epsilon(\bx) = \m_0(\bx) + \epsilon \m_1(\bx,\frac{\bx}{\epsilon}) + \epsilon^2\m_2(\bx,\frac{\bx}{\epsilon}),
\end{equation}
where $\m_0$ is the homogenized solution to \eqref{eqn:homogenized LLG system}, $\m_1$ is the first-order corrector defined in \eqref{eqn:sln m1}, and $\m_2$ is the second-order corrector determined by Theorem \ref{lemma m2}.
Then replacing $\m^\epsilon$ by $\widetilde{\m}^\epsilon$ in \eqref{eqn:LLG system form 3} provides the notation of consistence error $\boldsymbol{f}^\epsilon$:
\begin{equation}\label{eqn:equivalent system of m epsilon}
	\mathcal{L}_{\mathrm{LLG}}(\widetilde{\m}^\epsilon)
	=\boldsymbol{f}^\epsilon.
\end{equation}
Together \eqref{eqn:LLG system form 3} and \eqref{eqn:equivalent system of m epsilon}, we can obtain the equation of error $\e^\epsilon = \m^\epsilon - \widetilde{\m}^\epsilon$, denoted by
\begin{equation}\label{system of e}
	\partial_t\e^\epsilon - \mathcal{L}^\epsilon \e^\epsilon + \boldsymbol{f}^\epsilon
	=\boldsymbol{0},
\end{equation}
where $\mathcal{L}^\epsilon$ is second-order linear elliptic operator depending on $\m^\epsilon$ and $\widetilde{\m}^\epsilon$, that can be characterized as
\begin{equation}\label{def L eps}
	\mathcal{L}^\epsilon(\e^\epsilon) =
	\alpha \widetilde{\mathcal{H}}^\epsilon_e(\e^\epsilon) 
	-
	\mathbf{D}_1(\e^\epsilon) - \mathbf{D}_2(\e^\epsilon).
\end{equation}
Here $\widetilde{\mathcal{H}}^\epsilon_e(\m^\epsilon)$ is the linear part of $\mathcal{H}^\epsilon_e(\m^\epsilon) $, i.e.,
\begin{equation*}
	\begin{aligned}
		\widetilde{\mathcal{H}}^\epsilon_e(\m^\epsilon)
		& :=  \mathcal{H}^\epsilon_e(\m^\epsilon) -
		M^{\epsilon} \ha ,
	\end{aligned}
\end{equation*}
procession term $\mathbf{D}_1$ is calculated by
\begin{equation}\label{define D_1}
	\begin{aligned}
		\mathbf{D}_1 (\e^\epsilon)&= 
		\m^\epsilon \times \mathcal{H}^\epsilon_e(\m^\epsilon) 
		-
		\widetilde{\m}^\epsilon \times \mathcal{H}^\epsilon_e(\widetilde{\m}^\epsilon) \\
		&= 
		\m^\epsilon \times \widetilde{\mathcal{H}}^\epsilon_e(\e^\epsilon) 
		+
		\e^\epsilon \times \mathcal{H}^\epsilon_e(\widetilde{\m}^\epsilon),
	\end{aligned}
\end{equation}
and the degeneracy term $\mathbf{D}_2$ reads as
\begin{equation*}
	\begin{aligned}
		\mathbf{D}_2 (\e^\epsilon)&= 
		- \alpha  g_l^\epsilon[\m^\epsilon]
		\m^\epsilon
		+
		\alpha  g_l^\epsilon[\widetilde{\m}^\epsilon]
		\widetilde{\m}^\epsilon \\
		&= 
		- \alpha \big( \mathcal{B}^\epsilon[\e^\epsilon,\m^\epsilon ] 
		+ \mathcal{B}^\epsilon[\widetilde{\m}^\epsilon,\e^\epsilon ]  
		+ M^\epsilon (\ha\cdot \e^\epsilon) \big)\m^\epsilon
		- \alpha g_l^\epsilon[\widetilde{\m}^\epsilon]  \e^\epsilon.
	\end{aligned}
\end{equation*}

Moreover, we define a correctional error $\e^\epsilon_{{\mathrm{b}}}$ as
\begin{equation}\label{define e_b in intro}
	\e^\epsilon_{{\mathrm{b}}} = \e^\epsilon - \boldsymbol{\omega}_\mathrm{b},
\end{equation}
where $\boldsymbol{\omega}_\mathrm{b}$ is the boundary corrector satisfying $\boldsymbol{\omega}_\mathrm{b} = \boldsymbol{\omega}_\mathrm{N} - \boldsymbol{\omega}_\mathrm{M}$, for $\boldsymbol{\omega}_\mathrm{N}$ the Neumann corrector given in \eqref{define w_b}, and $\boldsymbol{\omega}_\mathrm{M}$ the modification determined in \eqref{define tilde omega}.
 Then equation \eqref{system of e} leads to
\begin{equation}\label{equation of e_b in intro}
	\partial_t\e^\epsilon_{{\mathrm{b}}} 
	- 
	\mathcal{L}^\epsilon \e^\epsilon_{\mathrm{b}} 
	+
	\big(\partial_t \boldsymbol{\omega}_\mathrm{b}
	-
	\mathcal{L}^\epsilon\boldsymbol{\omega}_\mathrm{b}
	+
	\boldsymbol{f}^\epsilon\big) =0.
\end{equation}
By the Lax equivalence theorem kind of argument, the estimate of error $\e^\epsilon_{{\mathrm{b}}}$ follows from consistency analysis of \eqref{eqn:equivalent system of m epsilon}, energy estimate of boundary corrector, and stability analysis of \eqref{equation of e_b in intro}.

\subsection{Proof of Theorem \ref{thm:1}}
\begin{proof}
Following the above notations,
for the consistency error $\boldsymbol{f}^\epsilon$, Theorem \ref{lemma:consistency m2} says that it can be divided as $\boldsymbol{f}^\epsilon = \boldsymbol{f}_0 + \widetilde{\boldsymbol{f}}$, satisfying $\Vert \widetilde{\boldsymbol{f}}(t) \Vert_{L^{2}(\Omega)}
\le C \epsilon$, and
\begin{equation*}
	\begin{aligned}
		&\Vert \boldsymbol{f}_0(t) \Vert_{L^{2}(\Omega)}
		= 0,\quad \text{when $\mu_0 = 0$},\\
		&\Vert \boldsymbol{f}_0(t) \Vert_{L^{r}(\Omega)}
		\le C_r \mu_0
		\big(\epsilon^{1/r}
		+
		\epsilon \ln (\epsilon^{-1} + 1)\big),\quad \text{when $\mu_0 > 0$, $n\neq1$},
	\end{aligned}
\end{equation*}
where constants $C_r$ and $C$ are independent of $\epsilon$, for any $t\in (0,T)$, and $1\le r < +\infty$.
Considering the boundary corrector terms in \eqref{equation of e_b in intro}, by Theorem \ref{theorem:omega} there exists $C = C ( \Vert \nabla \m^\epsilon \Vert_{L^{2}(\Omega)} )$ such that
\begin{equation*}\label{estimate of L w_b in intro}
	\begin{aligned}
		&\Vert \partial_t \boldsymbol{\omega}_{{\mathrm{b}}}(t) \Vert_{L^{2}(\Omega)}
		\le 
		C \epsilon \ln (\epsilon^{-1} + 1),\\
		&\Vert \mathcal{L}^\epsilon\boldsymbol{\omega}_\mathrm{b}(t)
		\Vert_{L^2(\Omega)}
		\le 
		C \epsilon [\ln (\epsilon^{-1} + 1)]^2
		+
		C\Vert \e^\epsilon_{{\mathrm{b}}} (t) \Vert_{H^{1}(\Omega)},
	\end{aligned}
\end{equation*}
for any $t\in (0,T)$.
As for initial-boundary data of $\e^\epsilon_{{\mathrm{b}}}$, using Theorem \ref{thm: Estimates of initial-boundary data} we write with $C = C ( \Vert \nabla \m^\epsilon \Vert_{L^{2}(\Omega)} )$,
\begin{align*}\label{initial condition in intro}
	\Vert \e^\epsilon_{{\mathrm{b}}}(\bx,0) \Vert_{H^{1}(\Omega)}
	+
	\Vert \frac{\partial}
	{\partial \boldsymbol{\nu}^\epsilon} \e^\epsilon_{{\mathrm{b}}} \Vert_{W^{1,\infty}(0,T; B^{-1/2, 2} (\partial\Omega))}
	\le C \epsilon \ln (\epsilon^{-1} + 1).
\end{align*}

Now let us turn to stability analysis of \eqref{equation of e_b in intro}. For the $L^\infty(0,T;L^2(\Omega))$ norm, let $\sigma=1$ when $n =1,2$, and $\sigma=6/5$ when $n =3$, we can apply Theorem \ref{theorem:stability} to derive for $n=1,2,3$
	\begin{equation}\label{ineq of e^eps_b}
		\begin{aligned}
			&\Vert \e^\epsilon_{{\mathrm{b}}} \Vert^2_{L^\infty(0,T;L^2(\Omega))} 
			+ 
			 \Vert \nabla \e^\epsilon_{{\mathrm{b}}} \Vert^2_{L^2(0,T; L^{2} (\Omega))}\\
			\le  &
			C_{\delta} \Big( 
			\Vert \e^\epsilon_{{\mathrm{b}}}(\bx,0) \Vert_{L^{2}(\Omega)}^2
			+
			\Vert \frac{\partial}
			{\partial \boldsymbol{\nu}^\epsilon} \e^\epsilon_{{\mathrm{b}}} \Vert_{L^{2}(0,T; B^{-1/2, 2} (\partial\Omega))}^2
			+
			\Vert \widetilde{\boldsymbol{f}} \Vert^2_{L^2(0,T; L^{2} (\Omega))}\\
			&\qquad +
			\Vert \partial_t \boldsymbol{\omega}_{{\mathrm{b}}} \Vert^2_{L^2(0,T; L^{2} (\Omega))}
			+
			\gamma(\epsilon)
			\Vert \boldsymbol{f}_0 \Vert^2_{L^2(0,T; L^{\sigma} (\Omega))} \Big) \\
			& + 
			\delta\Vert \mathcal{L}^\epsilon \boldsymbol{\omega}_{{\mathrm{b}}}  \Vert^2_{L^2(0,T; L^{2} (\Omega))}
			+
			\epsilon^2 \Vert \mathcal{A}_\epsilon \e^\epsilon_{{\mathrm{b}}} \Vert_{L^2(0,T; L^{2} (\Omega))}^2.
		\end{aligned}
	\end{equation}
	with 
	\begin{equation*}
		\left\{\begin{aligned}
			&\gamma(\epsilon)=1,&
			&\text{when $n =1,3$},\\
			&\gamma(\epsilon)=[\ln (\epsilon^{-1} + 1)]^2, & &\text{when $n =2$}.
		\end{aligned}
		\right.
	\end{equation*}
	Constant $C_\delta = C_\delta(\Vert \nabla \m^\epsilon \Vert_{L^{4}(\Omega)})$. 
	Now taking $\delta$ small enough in \eqref{ineq of e^eps_b}, and using the fact
	\begin{equation*}
		\Vert \mathcal{A}_\epsilon \e^\epsilon_{{\mathrm{b}}} \Vert_{L^2(0,T; L^{2} (\Omega))}
		\le 
		C \ln (\epsilon^{-1} + 1)
	\end{equation*}
	with $C = C\big(\Vert \mathcal{A}_\epsilon \m^\epsilon \Vert_{L^{2} (\Omega)}\big)$ from Theorem \ref{theorem:omega}, we finally obtain
	\begin{equation*}\label{final estimate of e_b}
		\begin{aligned}
			&\Vert \e^\epsilon_{{\mathrm{b}}} \Vert_{L^\infty(0,T;L^2(\Omega))} 
			\le 
			\left\{\begin{aligned}
				&\beta(\epsilon), & &\text{when $\mu_0 > 0$, $n=2,3$},\\
				&C \epsilon[\ln (\epsilon^{-1} + 1)]^2, & &\text{when $\mu_0 = 0$, $n=1,2,3$},
			\end{aligned}
			\right.
		\end{aligned}
	\end{equation*}
where $\beta(\epsilon)$ satisfies \eqref{value of beta}.
Using the fact $\m^\epsilon -  \m_0 = \e^\epsilon_{{\mathrm{b}}} + \epsilon\m_1 + \epsilon^2 \m_2 + \boldsymbol{\omega}_\mathrm{b}$, along with the estimates of $\epsilon\m_1$, $\epsilon^2 \m_2$, $\boldsymbol{\omega}_\mathrm{b}$ in Lemma \ref{lemma: estimate of chi}-\ref{lemma: estimate of theta and gamma}, we obtain the $L^2$ estimates in Theorem \ref{thm:1}.

As for the stability of \eqref{equation of e_b in intro} in $L^\infty(0,T;H^1(\Omega))$  norm, we can 
apply Theorem \ref{theorem:stability 2} to obtain for $n=1,2,3$
\begin{align*}
	\Vert \nabla \e^\epsilon_{{\mathrm{b}}} \Vert^2_{L^\infty(0,T;L^2(\Omega))} 
	\le 
	C \Big( \Vert \e^\epsilon_{{\mathrm{b}}}(\bx,0) \Vert_{H^{1}(\Omega)}^2
	+
	\Vert \frac{\partial}
	{\partial \boldsymbol{\nu}^\epsilon} \e^\epsilon_{{\mathrm{b}}} \Vert_{H^{1}(0,T; B^{-1/2, 2} (\partial\Omega))}^2\quad \\
	+
	\Vert\mathcal{L}^\epsilon \boldsymbol{\omega}^\epsilon_{{\mathrm{b}}} \Vert^2_{ L^2(0,T; L^{2} (\Omega))}
	+
	\Vert \boldsymbol{f}^\epsilon \Vert^2_{L^2(0,T; L^{2} (\Omega))}
	+
	\Vert \partial_t \boldsymbol{\omega}_{{\mathrm{b}}} \Vert^2_{L^2(0,T; L^{2} (\Omega))}
	\Big),
\end{align*}
where constant $C = C( \Vert \mathcal{A}_\epsilon \m^\epsilon \Vert_{L^{2} (\Omega)}, \Vert \nabla \m^\epsilon \Vert_{L^{4}(\Omega)})$. Together with above results, and estimate for $\Vert \nabla \e^\epsilon_{{\mathrm{b}}} \Vert^2_{L^2(0,T; L^{2} (\Omega))}$ in \eqref{ineq of e^eps_b}, we arrive at
\begin{equation*}
	\begin{aligned}
		&\Vert \nabla \e^\epsilon_{{\mathrm{b}}} (t)\Vert^2_{L^\infty(0,T; L^{2}(\Omega))} 
		\le 
		\left\{\begin{aligned}
			&C \epsilon^{1/2}, & &\text{when $\mu_0 > 0$, $n=2,3$},\\
			&C \epsilon[\ln (\epsilon^{-1} + 1)]^2, & &\text{when $\mu_0 = 0$, $n=1,2,3$},
		\end{aligned}
		\right.
	\end{aligned}
\end{equation*}
	by the representation of $\e^\epsilon_{{\mathrm{b}}}$ in \eqref{rewrite e_b}, 
	together with estimate of $\m_2$ and $\boldsymbol{\omega}_\mathrm{M}$ in Lemma \ref{lemma: estimate of theta and gamma},
	it leads to the $H^1$ estimates in Theorem \ref{thm:1}.

Notice that all the constants in our estimate depend on the value of $\Vert \mathcal{A}_\epsilon \m^\epsilon(t) \Vert_{L^{2} (\Omega)}$ and $\Vert \nabla \m^\epsilon(t) \Vert_{L^{4}(\Omega)}$, which from Theorem \ref{regularety 3} are uniformly bounded with respect to $\epsilon$ and $t$ for any $t\in (0,T^*)$, with some $T^*\in(0, T]$. This completes the proof.
\end{proof}

\section{The Asymptotic Expansion}\label{expansion}
In this section, we derive the second-order corrector using the formal asymptotic expansion. First, let us define the first-order corrector $\m_1$ by
\begin{equation}\label{eqn:sln m1}
\m_1(\bx,\by) = \sum_{j= 1}^n \chi_j(\by)\frac{\partial}{\partial x_j}\m_0(\bx),
\end{equation}
where $\chi_j$, $j = 1,\dots,n$ are auxiliary functions satisfying cell problem
\begin{equation}\label{eqn:chi_j}
\left\{ \begin{aligned}
& \bdiv \big( a(\by)\nabla \chi_j(\by)\big) = - \sum_{i= 1}^n \frac{\partial}{\partial y_i} a_{ij}(\by),\\
& \chi_j \quad \mbox{$Y$-periodic},
\end{aligned} \right.
\end{equation}
such that the first geometric property in \eqref{geo property} holds. As for 
the second-order corrector $\m_2$, we assume it as a two-scale function satisfying
\begin{equation*}
\left\{ \begin{aligned}
&\m_2(\bx,\by) \mbox{ is defined for $\bx\in \Omega$ and $\by\in Y$},\\
&\m_2(\cdot, \by) \mbox{ is $Y$-periodic}.
\end{aligned} \right.
\end{equation*}

For notational convenience, given a two-scale function in the form of $\m(\bx, \frac{\bx}{\epsilon})$, we denote the fast variable $\by = \frac{\bx}{\epsilon}$ and have the following chain rule
\begin{equation}\label{notation of chain rule}
\nabla \m(\bx,\frac{\bx}{\epsilon}) = [(\nabla_{\bx} + \epsilon^{-1}\nabla_{\by}) \m](\bx,\by).
\end{equation}
Moreover, denoting $\mathcal{A}_\epsilon =  \bdiv ( a^\epsilon\nabla )$, one can rewrite
\begin{equation*}
\mathcal{A}_\epsilon \m(\bx,\frac{\bx}{\epsilon})
= [(\epsilon^{-2}\mathcal{A}_0 + \epsilon^{-1}\mathcal{A}_1 + \mathcal{A}_2) \m ] (\bx,\by),
\end{equation*}
where 
\begin{equation}\label{eqn:A0 A1 and A2}
\left\{ \begin{aligned}
& \mathcal{A}_0 = \bdiv_{\by} \big( a(\by)\nabla_{\by} \big),\\
& \mathcal{A}_1 = \bdiv_{\bx} \big( a(\by)\nabla_{\by} \big) 
+ \bdiv_{\by} \big( A(\by)\nabla_{\bx} \big),\\
& \mathcal{A}_2 = \bdiv_{\bx} \big( a(\by)\nabla_{\bx} \big).
\end{aligned} \right.
\end{equation}

The procedure to determine $\m_2$ is standard. With the notation in \eqref{def approximate sln}, assume $\m^\epsilon$ can be written in form of
\begin{equation}\label{eqn:formal expansion}
\m^\epsilon(\bx) = \widetilde{\m}^\epsilon(\bx) + o(\epsilon^2).
\end{equation}
One can derive $\m_2$ by substituting \eqref{eqn:formal expansion} into \eqref{eqn:LLG system} and comparing like terms of $\epsilon$. However, it is a bit fussy in the presence of stray field. Let us outline the main steps here. Revisiting the stray field $\hd[M^\epsilon \m^\epsilon(x)] = \nabla U^\epsilon$ in \eqref{define hd}-\eqref{eqn:stray field eqn}, one finds that the potential function $U^\epsilon = U^\epsilon[M^\epsilon \m^\epsilon(x)]$ satisfies
\begin{equation}\label{eqn: U^epsilon}
\begin{aligned}
\Delta U^\epsilon 
= &
- \bdiv (M_s(\frac{\bx}{\epsilon})\m^\epsilon\mathcal{X}_{\Omega}).
\end{aligned}
\end{equation}
Substituting $ U^\epsilon = \Sigma_{j =0}^2 \epsilon^jU_j(\bx,\frac{\bx}{\epsilon}) + o(\epsilon^2)$ and \eqref{eqn:formal expansion} into \eqref{eqn: U^epsilon} and combining like terms of $\epsilon$ leads to 
\begin{equation}\label{eqn:expansion of stray field}
\left\{ \begin{aligned}
&\mathrm{div}_{\by} (\nabla_{\by} U_0(\bx,\by)) = 0,\\
&\mathrm{div}_{\by} (\nabla_{\by} U_1(\bx,\by)) = - \bdiv_{\by} (M_s(y)\m_0(\bx)\mathcal{X}_{\Omega}(\bx)),\\
&\mathrm{div}_{\bx} (\nabla_{\bx} U_0(\bx,\by)) + 2\mathrm{div}_{\by} (\nabla_{\bx} U_1(\bx,\by))
+ \mathrm{div}_{\by} (\nabla_{\by} U_2(\bx,\by))\\
&\quad =- M_s(\by)\mathrm{div}_{\bx}\m_0(\bx)\mathcal{X}_{\Omega}(\bx) - \mathrm{div}_{\by} (M_s(\by) \m_1(\bx,\by)\mathcal{X}_{\Omega}(\bx)).
\end{aligned} \right.
\end{equation}
The first equation in \eqref{eqn:expansion of stray field} implies that $U_0(\bx,\by) = U_0(\bx)$ since the Lax-Milgram Theorem ensures the uniqueness and existence of solution (up to a constant). Integrating the third equation in \eqref{eqn:expansion of stray field} with respect to $\by$ yields
\begin{equation*}
\Delta U_0(\bx) = -
\mathrm{div}(M^0\m_0\mathcal{X}_{\Omega}).
\end{equation*}
An application of \eqref{define hd}-\eqref{eqn:stray field eqn} implies that $U_0$ is actually the potential function of $\hd[M^\mathrm{h} \m_0]$, i.e.,
\begin{equation}\label{eqn: relation between U0 and hd}
\nabla U_0(\bx) = 
\hd[M^\mathrm{h} \m_0]
=
M^\mathrm{h} \hd[\m_0].
\end{equation}
With notation given in \eqref{define H_d},
one can deduce from the second equation in \eqref{eqn:expansion of stray field} that $\m_0(\bx)\mathcal{X}_{\Omega}(\bx) \mathbf{U}(\by) = U_1(\bx,\by)$ up to a constant in the $H^1(Y)$ space. Hence it follows that by \eqref{define H_d}
\begin{equation}\label{eqn: relation between U1 and Hd}
\begin{aligned}
\nabla_{\by} U_1(\bx,\by) =& \mathcal{X}_{\Omega}(\bx) \m_0(\bx)\cdot \mathbf{H}_{\mathrm{d}}[M_s(\by)](\by).
\end{aligned}
\end{equation}
Substituting \eqref{eqn: relation between U0 and hd} and \eqref{eqn: relation between U1 and Hd} into the expansion of $U^\epsilon$, one can deduce that, for $\bx\in \Omega$,
\begin{equation}\label{eqn:stray field order 0}
\begin{aligned}
\hd[M^\epsilon \m^\epsilon]
=
\nabla U^\epsilon
=&
\hd[M^\mathrm{h} \m_0] 
+ \m_0(\bx) \cdot \mathbf{H}_{\mathrm{d}}[M_s(\by)](\frac{\bx}{\epsilon})
+ O(\epsilon).
\end{aligned}
\end{equation}

Substituting \eqref{eqn:formal expansion}, \eqref{eqn:sln m1}, \eqref{eqn:A0 A1 and A2}, \eqref{eqn:stray field order 0} into \eqref{eqn:LLG system} and collecting terms of $O(\epsilon^0)$, we obtain the following equations
\begin{equation}\label{eqn:system m2}
\left\{ \begin{aligned}
&\begin{gathered} \partial_t\m_0 - \alpha \m_0 \times \partial_t\m_0 = 
- (1+\alpha^2) \m_0 \times \left\{ \mathcal{A}_0 \m_2 + \mathcal{H}^{\mathrm{a}}_e  \right\},
\end{gathered}\\
& \m_2 \quad \mbox{$Y$-periodic in $\by$},
\end{aligned} \right.
\end{equation}
where
\begin{equation}\label{define Ha}
\begin{aligned}
\mathcal{H}^{\mathrm{a}}_e
= &
\mathcal{A}_1 \m_1 + \mathcal{A}_2 \m_0 - K^\epsilon(\m_0\cdot \bu)\bu \\
&+ 
\mu_0 M_s  \hd[M^\mathrm{h} \m_0] + \mu_0 M_s \m_0\cdot \mathbf{H}_{\mathrm{d}}[M_s(\by)] + M_s \ha.
\end{aligned}
\end{equation}
Substituting \eqref{eqn:homogenized LLG system} into \eqref{eqn:system m2} leads to
\begin{equation}\label{eqn:system m2 form2}
\left\{ \begin{aligned}
&\begin{gathered} 
\m_0 \times  \mathcal{A}_0 \m_2
= 
\m_0 \times \left\{  \mathcal{H}^0_e(\m_0) 
- \mathcal{H}^{\mathrm{a}}_e \right\},
\end{gathered}\\
& \m_2 \quad \mbox{$Y$-periodic in $\by$}.
\end{aligned} \right.
\end{equation}
\eqref{eqn:system m2 form2} is the degenerate system that determines $\m_2$ in terms of $\m_0$.

\subsection{Second-order corrector}
The well-posedness of \eqref{eqn:system m2 form2} is nontrivial due to the degeneracy. In the following Theorem, by searching a suitable solution satisfying \eqref{geo property}, we give the existence result, and derive an explicit expression of $\m_2$ in terms of $\m_0$ and some auxiliary functions.
\begin{theorem}\label{lemma m2}
	Given $\m_0 \in L^{\infty}\left([0,T];H^2(\Omega)\right)$ the homogenization solution and $\m_1$ calculated in \eqref{eqn:sln m1},  define 
	\begin{equation*}
	\mathcal{V} = \Big\{ \m\in H^2(Y)\cap H^1_{\mathrm{per}}(Y) :\ \m \cdot \m_0 =-\frac{1}{2} \abs{\m_1}^2 \quad \mbox{a.e. in $\Omega\times Y$} \Big\},
	\end{equation*}
	then \eqref{eqn:system m2 form2} admits a unique solution $\m_2(\bx, \by) \in \mathcal{V}/T_{\m_0}(S^2)$, with notation $T_{\m_0}(S^2)$ denoting the tangent space of $\m_0$.
\end{theorem}
\begin{proof}
	Assume $\m_2(\bx, \by) \in \mathcal{V}$, i.e., $\m_2 \cdot \m_0 =-\frac{1}{2} \abs{\m_1}^2$. Applying $\mathcal{A}_0$ to both sides of it yields
	\begin{equation}\label{eqn:substituting of m2}
	\m_0\cdot \mathcal{A}_0 \m_2
	= - a(\by) \nabla_{\by}\m_1 \cdot \nabla_y \m_1 - \m_1\cdot \mathcal{A}_0 \m_1.
	\end{equation}
	Taking the cross-product with $\m_0$ to \eqref{eqn:system m2 form2} and substituting \eqref{eqn:substituting of m2} lead to 
	\begin{equation}\label{eqn:system m2 form3}
	\begin{aligned} 
	\mathcal{A}_0 \m_2 =& - \{ \mathcal{H}^{\mathrm{a}}_e - \mathcal{H}^0_e(\m_0) \}
	+
	\big\{ \m_0 \cdot \big( \mathcal{H}^{\mathrm{a}}_e - \mathcal{H}^0_e(\m_0)\big) \big\} \m_0\\
	& - ( \m_1\cdot \mathcal{A}_0 \m_1 
	+ a(\by) \nabla_{\by}\m_1 \cdot \nabla_{\by} \m_1 ) \m_0.
	\end{aligned}
	\end{equation}
	Now using the fact
	\begin{equation}\label{H0 - Ha}
	\int_Y \mathcal{H}^0_e(\m_0) 
	- \mathcal{H}^{\mathrm{a}}_e \d \by = 0,
	\end{equation}
	together with \eqref{eqn:substituting of m2}, one can check equation \eqref{eqn:system m2 form3} satisfies the compatibility condition for $Y$-periodic function $\m_2$ in $\by$. Thus by the application of Lax-Milgram Theorem and smoothness assumption, \eqref{eqn:system m2 form3} admits a unique regular solution up to a function independent of $\by$, denoted by $\m_2(\bx,\by) + \widetilde{\m}_2(\bx)$. Moreover, one can determine $\widetilde{\m}_2(\bx)$  by
	\begin{equation*}
		\m_0\cdot (\m_2 + \widetilde{\m}_2) = -\frac{1}{2} \abs{\m_1}^2,
	\end{equation*}
	such that $\m_2 + \widetilde{\m}_2\in \mathcal{V}$, and therefore is also a solution to \eqref{eqn:system m2 form2} by taking the above transformation inversely. 

%
\end{proof}
\begin{remark}
	One can check that equation \eqref{eqn:system m2 form3} has a solution
	\begin{equation}\label{eqn:sln m2}
		\m_2 = \sum_{i,j= 1}^n \theta_{ij} 
		\frac{\partial^2 \m_0}{\partial x_i \partial x_j} 
		+ \sum_{i,j= 1}^n (\theta_{ij} + \small\frac{1}{2}\chi_i\chi_j) \left(\frac{\partial \m_0}{\partial x_i} \cdot \frac{\partial \m_0}{\partial x_j} \right) \m_0
		+ \mathcal{T}_{\mathrm{low}}
		- (\m_0\cdot \mathcal{T}_{\mathrm{low}})\m_0
	\end{equation}
with low-order terms $\mathcal{T}_{\mathrm{low}}$ calculated by
\begin{equation*}
	\mathcal{T}_{\mathrm{low}} = 
	- \kappa (\m_0\cdot \bu)\bu
	+
	\mu_0 \rho \hd[M^\mathrm{h} \m_0] 
	+ \mu_0 \m_0\cdot \boldsymbol{\Lambda}
	+ M_s \ha,
\end{equation*}
	where $\theta_{ij}$ and $\kappa$, $\rho$, $\boldsymbol{\Lambda}$ are given by 
	\begin{equation}\label{eqn:theta kappa and gamma}
		\left\{ \begin{aligned}
			&\mathcal{A}_0\theta_{ij} = 
			a^0_{ij} - \big(a_{ij} + \sum_{k=1}^n a_{ik}\frac{\partial\chi_j}{\partial y_k}\big)
			- \sum_{k=1}^n \frac{\partial (a_{ik}\chi_j)}{\partial y_k}
			,\\
			&\mathcal{A}_0 \rho  = M_s (\by)-M^0,\quad
			\mathcal{A}_0\kappa = K(\by)-K^0,\\
			&\mathcal{A}_0 \boldsymbol{\Lambda} =
			M_s(\by) \mathbf{H}_{\mathrm{d}}[M_s(\by)](\by) - \mathbf{H}_{\mathrm{d}}^0,\\
			&\theta_{ij},\  \kappa,\ \rho, \ \boldsymbol{\Lambda}, \quad \mbox{are $Y$-periodic}.
		\end{aligned} \right.
	\end{equation}
	Moreover, one can find $\m_2$ defined above satisfies geometric property \eqref{geo property}, therefore is also the solution to equation \eqref{eqn:system m2 form2}. In the following, we may assume second-order correct $\m_2$ takes the form in \eqref{eqn:sln m2}.
\end{remark}

\section{Consistency Estimate}\label{consistency}
In this section, we aim to estimate the consistence error $\boldsymbol{f}^\epsilon$ defined in \eqref{eqn:equivalent system of m epsilon}.
Following the notation in \eqref{notation of chain rule}-\eqref{eqn:A0 A1 and A2}, by the definition of $\widetilde{\m}^\epsilon$, \eqref{eqn:equivalent system of m epsilon} can be written in terms of 
\begin{equation*}
\boldsymbol{f}^\epsilon = 
\epsilon^{-2}\boldsymbol{f}_{-2} + \epsilon^{-1}\boldsymbol{f}_{-1} + \boldsymbol{f}_{0} + \epsilon \boldsymbol{f}_{1} + \epsilon^{2}\boldsymbol{f}_{2}.
\end{equation*}
It is easy to check that $\boldsymbol{f}_{-2} = \boldsymbol{f}_{-1} = \boldsymbol{0}$ by the definition of $\m_0$, $\m_1$ in Section \ref{expansion}. Along the same line, by the H\"{o}lder's inequality, one has 
\begin{equation*}
\Vert \boldsymbol{f}_{1} \Vert_{L^{2}(\Omega)} + \Vert \boldsymbol{f}_{2} \Vert_{L^{2}(\Omega)}
\le C,
\end{equation*} 
where $C$ depends on the $L^{2}(\Omega)$ and $L^{\infty}(\Omega)$ norms of $ \m_i(\bx, \frac{\bx}{\epsilon})$, $\nabla_{\bx}\m_i(\bx, \frac{\bx}{\epsilon})$, $\nabla_{\by}\m_i(\bx, \frac{\bx}{\epsilon})$, $i = 0,1,2$, and thus is bounded from above by $\Vert \nabla \m_0 \Vert_{H^{4}(\Omega)}$
with the help of smoothness assumption and Sobolev inequality.

It remains to estimate $\boldsymbol{f}_0$, let us prove that $\boldsymbol{f}_0$ only depends on the consistence error of stray field, by the help of geometric property \eqref{geo property}.
Denote the consistence error of stray field by
\begin{equation}\label{define tilde h}
\begin{aligned} 
&\widetilde{\mathbf{h}}
=
\mu M^\epsilon  \mathbf{h}_\mathrm{d}[(M^\epsilon - M^\mathrm{h}) \m_0]
- \mu M^\epsilon   \mathbf{H}_{\mathrm{d}}[M_s(\by)](\frac{\bx}{\epsilon})\cdot \m_0
\end{aligned}
\end{equation}
with $\mathbf{H}_{\mathrm{d}}$ given in \eqref{define H_d}. After some algebraic calculations and the usage of \eqref{eqn:system m2} and \eqref{define Ha}, one has
\begin{equation}\label{eqn:value of f0}
\begin{aligned}
\boldsymbol{f}_{0} =& \partial_t \m_0 - \alpha \left\{ \mathcal{A}_0 \m_2 + \mathcal{H}^{\mathrm{a}}_e + \widetilde{\mathbf{h}} \right\} 
+ \m_0 \times \left\{ \mathcal{A}_0 \m_2 + \mathcal{H}^{\mathrm{a}}_e + \widetilde{\mathbf{h}} \right\}  \\
&- \alpha  g_l^\epsilon[\m_0]\m_0 - (a^\epsilon \nabla_y\m_1 \cdot \nabla_y\m_1)\m_0 - 2(a^\epsilon \nabla_y\m_1 \cdot \nabla\m_0)\m_0.
\end{aligned}
\end{equation}
Notice that the classical solution $\m_0$ to \eqref{eqn:homogenized LLG system} also satisfies the equivalent form
\begin{gather}\label{eqn:equivalent system of m0}
\partial_t\m_0 - \alpha \mathcal{H}^0_e(\m_0) 
+ \m_0 \times \mathcal{H}^0_e(\m_0)
- \alpha
g_l^0[\m_0] \m_0
= 0,
\end{gather}
where 
\begin{equation*}
\begin{aligned}
g_l^0[\m] :=&  a^0 \vert \nabla \m\vert^2 
+ 
K^0 \left( \m\cdot \bu \right)\bu 
-
\mu_0  (M^0)^2 \hd [\m]\cdot \m\\
& -
\mu_0
\m\cdot \mathbf{H}^0_{\mathrm{d}} \cdot \m
-
\ha\cdot M^0 \m .
\end{aligned}
\end{equation*}
Substituting \eqref{eqn:equivalent system of m0} into \eqref{eqn:value of f0} and using \eqref{eqn:system m2 form3} lead to
\begin{align}
\boldsymbol{f}_{0} =& 
- \alpha \widetilde{\mathbf{h}}
+ 
\m_0 \times \widetilde{\mathbf{h}} 
-\alpha\big\{ \m_0 \cdot \big( \mathcal{H}^{\mathrm{a}}_e - \mathcal{H}^0_e(\m_0)\big) \big\} \m_0 \nn\\
& + \alpha( \m_1\cdot \mathcal{A}_0 \m_1 
- 2a^\epsilon \nabla_y\m_1 \cdot \nabla\m_0 ) \m_0
+ 
\alpha g_l^0[\m_0] \m_0  - \alpha  g_l^\epsilon[\m_0]\m_0.\label{eqn:value of f0 form1}
\end{align}
Note that $\mathcal{A}_2 \m_0 = \mathcal{A}_\epsilon \m_0- \epsilon^{-1} \mathcal{A}_1 \m_0$,
one can deduce
\begin{equation*}
\begin{aligned} 
&\mathcal{H}^{\mathrm{a}}_e  
= 
\mathcal{H}^\epsilon_e(\m_0)
+
\mathcal{A}_1 \m_1 +
\epsilon^{-1} \mathcal{A}_0 \m_1
-
\widetilde{\mathbf{h}}.
\end{aligned}
\end{equation*}
Substituting it into \eqref{eqn:value of f0 form1}, and using the fact
\begin{equation*}
\m_0\cdot \mathcal{H}^\epsilon_e (\m_0) = - g_l^\epsilon[\m^0 ],\quad
\m_0\cdot \mathcal{H}^0_e(\m_0) = -g_l^0[\m^0 ],
\end{equation*}
one has
\begin{equation}\label{eqn:value of f0 form3}
\begin{aligned} 
\boldsymbol{f}_{0} =& 
- \alpha \widetilde{\mathbf{h}}
+ 
\m_0 \times \widetilde{\mathbf{h}} 
-\alpha\big\{ \m_0 \cdot \big( \mathcal{A}_1 \m_1 +
\epsilon^{-1} \mathcal{A}_0 \m_1
-
\widetilde{\mathbf{h}}\big) \big\} \m_0\\
& + \alpha( \m_1\cdot \mathcal{A}_0 \m_1 
- 2a^\epsilon \nabla_y\m_1 \cdot \nabla\m_0 ) \m_0.
\end{aligned}
\end{equation}
Apply $\mathcal{A}_0$ and $\mathcal{A}_1$ to both sides of $\m_0 \cdot\m_1 = 0$ respectively, and substitute resulting equations into \eqref{eqn:value of f0 form3}. After simplification, we finally obtain
\begin{equation}\label{eqn:value of f0 form2}
\begin{aligned} 
\boldsymbol{f}_{0} =& 
- \alpha \widetilde{\mathbf{h}}
+ 
\m_0 \times \widetilde{\mathbf{h}} 
+
\alpha\left( \m_0 \cdot 
\widetilde{\mathbf{h}} \right) \m_0.
\end{aligned}
\end{equation}

\eqref{eqn:value of f0 form2} implies that the convergence of $\boldsymbol{f}^\epsilon$ depends on the convergence of stray field error $\widetilde{\mathbf{h}}$. In fact, we have 
\begin{lemma}\label{consistency of stray field first part}
	For any $1\le r< \infty$, and $n = 2,3$, it holds that
	\begin{equation}\label{estimate of stray field}
	\begin{aligned} 
	\big\Vert M^\epsilon  \hd[(M^\epsilon -M^\mathrm{h}) \m_0] - M^\epsilon   \mathbf{H}_{\mathrm{d}}[M_s(\by)]&(\frac{\bx}{\epsilon})\cdot \m_0
	\big\Vert_{L^{r}(\Omega)}\\
	&\le 
	C 
	\epsilon^{1/r}
	+
	C
	\epsilon \ln (\epsilon^{-1} + 1),
	\end{aligned}
	\end{equation}
	where $C$ depends on $\Vert \nabla \m_0\Vert_{W^{1,\infty}(\Omega)}$, $\Vert \nabla M_s(\by) \Vert_{H^{1}(Y)}$ and is independent of $\epsilon$.
\end{lemma}
Proof of Lemma \ref{consistency of stray field first part} will be given in Section \ref{sec:Consistency of stray field}.  Lemma \ref{consistency of stray field first part} directly leads to the consistency error:
\begin{theorem}\label{lemma:consistency m2}(Consistency)	
	Given $\boldsymbol{f}^\epsilon$ defined in \eqref{eqn:equivalent system of m epsilon},
	it can be divided as $\boldsymbol{f}^\epsilon = \boldsymbol{f}_0 + \widetilde{\boldsymbol{f}}$, satisfying $\Vert \widetilde{\boldsymbol{f}} \Vert_{L^{2}(\Omega)}
	\le C \epsilon$, and
	\begin{equation}\label{eqn:consistency m2}
		\begin{aligned}
			&\Vert \boldsymbol{f}_0 \Vert_{L^{2}(\Omega)}
			= 0,\quad \text{when $\mu_0 = 0$},\\
			&\Vert \boldsymbol{f}_0 \Vert_{L^{r}(\Omega)}
			\le C_r \mu_0
			\big(\epsilon^{1/r}
			+
			\epsilon \ln (\epsilon^{-1} + 1)\big),\quad \text{when $\mu_0 > 0$, $n\neq1$},
		\end{aligned}
	\end{equation}
	for any $1\le r < +\infty$.
Here constant $C$ and $C_r$ depend on $\Vert \nabla \m_0\Vert_{H^4(\Omega)}$, $\Vert \nabla M_s(\by) \Vert_{H^{1}(Y)}$, and are independent of $\epsilon$.
\end{theorem}

\subsection{Estimate of some singular integral}\label{sec:Estimate of singular integral}
The strategy to prove Lemma \ref{consistency of stray field first part} is to rewrite the stray field into derivatives of Newtonian potential, thus the consistency estimate turns into the estimate of singular integrals. The following Lemmas introduce the estimate of singular integral in terms of distribution function and boundary layer. We will use the cut-off function $\eta^\epsilon$ within the interior of area away from boundary:
\begin{equation}\label{cut off function}
	\left\{\begin{aligned}
		&0 \le \eta^\epsilon\le 1, & \abs{\nabla \eta^\epsilon} \le C \epsilon^{-1},& \\
		&\eta^\epsilon(x) = 1 & \mbox{if $\mathrm{dist}(x, \partial\Omega)\ge \frac{2}{3}\epsilon$},&\\
		&\eta^\epsilon(x) = 0 & \mbox{if $\mathrm{dist}(x, \partial\Omega)\le \frac{1}{3}\epsilon$}.&
	\end{aligned}\right.
\end{equation}
where $\mathrm{dist}(\bx,\partial\Omega)$ denotes the distance between $\bx$ and $\partial\Omega$, and cut-off function $\phi^\epsilon$ in a small ball:
\begin{equation}\label{cut off function in small ball}
	\left\{\begin{aligned}
		&0 \le \phi^\epsilon\le 1, & \abs{\nabla \phi^\epsilon} \le C \epsilon^{-1},& \\
		&\phi^\epsilon(\bx) = 1 & \mbox{if $\abs{\bx}\le \frac{1}{3}\epsilon$},&\\
		&\phi^\epsilon(\bx) = 0 & \mbox{if $\abs{\bx}\ge \frac{2}{3}\epsilon$}.&
	\end{aligned}\right.
\end{equation}
Denote the boundary layer $\Omega^\epsilon$ as
\begin{equation*}
	\Omega^\epsilon = 
	\{ \bx \in \Omega, \mathrm{dist}(\bx, \partial\Omega)\le \epsilon \}.
\end{equation*}
\begin{lemma}\label{lemma: singula integral}
	Assume that scalar functions $f(\by)\in C^1(R^n)$ is $Y$-periodic, $g(\bx)\in C^1(\bar{\Omega})$, define for $\bx\in\Omega$
	\begin{equation*}
	u(\bx) = 	\int_{\Omega}
	\frac{\big\vert f(\frac{\bx}{\epsilon}) - f(\frac{\bz}{\epsilon})\big\vert} {\vert \bx-\bz \vert^{n}} 
	\d \bz,
	\quad\quad
	v(\bx) = 	\int_{\Omega^\epsilon}
	\frac{\vert g(\bx) - g(\bz)\vert} {\vert \bx-\bz \vert^{n}} 
	\d \bz,
	\end{equation*} 
	then $u(\bx) \in L^\infty(\Omega)$ logarithmically grows with respect to $\epsilon$, satisfying
	\begin{equation*}
		\begin{aligned}
			\Vert u \Vert_{L^{\infty}(\Omega)}
			\le C
			\ln (\epsilon^{-1} + 1) \Vert f(\by) \Vert_{L^{\infty}(Y)} 
			+
			C\Vert \nabla f(\by) \Vert_{L^{\infty}(Y)} ,
		\end{aligned}
	\end{equation*}
and $v(\bx)\in L^r(\Omega)$ decreases at speed of $O(\epsilon^{1/r})$ for any $1\le r<\infty$, satisfying
\begin{equation*}
	\begin{aligned}
		\Vert v \Vert_{L^{r}(\Omega)}
		\le C \epsilon^{1/r}
		\ln (\epsilon^{-1} + 1) 
		\big(\Vert g(\bx) \Vert_{L^{\infty}(\Omega)} 
		+
		\epsilon\Vert \nabla g(\bx) \Vert_{L^{\infty}(\Omega)}\big) .
	\end{aligned}
\end{equation*}
Constant $C$ is independent of $\epsilon$.
\end{lemma}
\begin{proof} Splitting the integral in $u$ into $\int_{\Omega-B(x,\epsilon)} + \int_{B(x,\epsilon)}$, one can estimate it by
	\begin{equation*}
		\begin{aligned}
			\abs{u(\bx)}
			\le & C
			\int_{\Omega-B(x,\epsilon)}
			\frac{ \Vert f(\by) \Vert_{L^{\infty}(Y)} }{\vert \bx-\bz \vert^{n}} \d \bz
			+ C \epsilon^{-1}
			\int_{B(x,\epsilon)}
			\frac{\Vert \nabla f(\by) \Vert_{L^{\infty}(Y)} }{\vert \bx-\bz \vert^{n-1}} 
			\d \bz,
		\end{aligned}
	\end{equation*}
	therefore the estimate of $u$ in Lemma follows by simple integral.
As for the estimate of $v$, by application of cut-off function $\phi^\epsilon = \phi^\epsilon (\bx-\bz)$, one has
\begin{equation*}
	\begin{aligned}
		\abs{v(\bx)}
		=&
		\int_{\Omega^\epsilon}
		\frac{\phi^\epsilon \vert g(\bx) - g(\bz)\vert} {\vert \bx-\bz \vert^{n}} 
		\d \bz
		+
			\int_{\Omega^\epsilon}
		\frac{(1-\phi^\epsilon)\vert g(\bx) - g(\bz)\vert} {\vert \bx-\bz \vert^{n}} 
		\d \bz\\
		\le & C
		\Vert \nabla g \Vert_{L^\infty(\Omega)}
		\int_{\Omega^\epsilon}
		\frac{\phi^\epsilon } {\vert \bx-\bz \vert^{n-1}} 
		\d \bz
		+ C
		\Vert g \Vert_{L^\infty(\Omega)}
		\int_{\Omega^\epsilon}
		\frac{1-\phi^\epsilon} {\vert \bx-\bz \vert^{n}} 
		\d \bz\\
		=& R_1 + R_2.
	\end{aligned}
\end{equation*}
For $R_1$, one can write by Fubini's Theorem
\begin{equation*}
	\begin{aligned}
		\Vert R_1 \Vert_{L^{r}(\Omega)}^r 
		\le&
		C \Vert \nabla g \Vert_{L^\infty(\Omega)}^r
		\int_{\Omega}\Big(\int_{\Omega^\epsilon}
		\frac{\phi^\epsilon}{\vert \bx-\bz \vert^{n - 1}}  \d \bz\Big)^r\d \bx\\
		\le &
		C \Vert \nabla g \Vert_{L^\infty(\Omega)}^r
		\sup\limits_{x\in\Omega}\Big(\int_{\Omega^\epsilon}
		\frac{\phi^\epsilon}{\vert \bx-\bz \vert^{n - 1}}  \d \bz\Big)^{r-1}\\
		&\quad\quad\times
		\sup\limits_{\bz\in\Omega^\epsilon}\int_{\Omega} 
		\frac{\phi^\epsilon}{\vert \bx-\bz \vert^{n - 1}}  \d \bx 
		\int_{\Omega^\epsilon}1 \d \bz\\
		\le & C \Vert \nabla g \Vert_{L^\infty(\Omega)}^r \cdot C \epsilon^{r-1}
		\cdot C \epsilon\cdot C\epsilon.
	\end{aligned}
\end{equation*}
As for $R_2$, applying the same argument leads to
\begin{equation*}
	\begin{aligned}
		\Vert R_2 \Vert_{L^{r}(\Omega)}^r 
		\le &
		C \Vert g \Vert_{L^\infty(\Omega)}^r
		\sup\limits_{x\in\Omega}\Big(\int_{\Omega^\epsilon}
		\frac{1-\phi^\epsilon}{\vert \bx-\bz \vert^{n}}  \d \bz\Big)^{r-1}\\
		&\quad\quad\times
		\sup\limits_{\bz\in\Omega^\epsilon}\int_{\Omega} 
		\frac{1-\phi^\epsilon}{\vert \bx-\bz \vert^{n}}  \d \bx 
		\int_{\Omega^\epsilon}1 \d \bz\\
		\le & C \Vert g \Vert_{L^\infty(\Omega)}^r \cdot C [\ln(\epsilon^{-1}+1)]^{r-1}
		\cdot C \ln(\epsilon^{-1}+1) \cdot C\epsilon.
	\end{aligned}
\end{equation*}
\end{proof}

\begin{lemma}\label{Newtonian on the boundary layer}
	Assume that a scalar function $f^\epsilon(\bx)\in L^\infty(\Omega)$ satisfies $f^\epsilon(\bx) = 0$ when $\bx\in \Omega-\Omega^\epsilon$, which means $f^\epsilon$ is nonzero only in boundary layer. Let $w(\bx)$ be the Newtonian potential of $ f^\epsilon$ in $\Omega$, i.e.,
	\begin{equation*}
		w(\bx) = \int_{\Omega} \Phi(\bx-\boldsymbol{z}) f^\epsilon(\boldsymbol{z}) \d \boldsymbol{z},\quad \bx\in\Omega,
	\end{equation*}
	where $\Phi$ is the fundamental solution of Laplace's equation. Then $w(\bx) \in W^{2,p}(\Omega)$ satisfies for any $1 \le p < +\infty$
	\begin{equation*}
		\begin{aligned}
			&\Vert \nabla^2 w \Vert_{L^{p}(\Omega)}
			\le C \big(\epsilon^{1/p}
			+
			\epsilon \ln (\epsilon^{-1} + 1) \big)
			\big(\Vert f^\epsilon(\bx) \Vert_{L^{\infty}(\Omega)} 
			+
			\epsilon\Vert \nabla f^\epsilon(\bx) \Vert_{L^{\infty}(\Omega)}\big).
		\end{aligned}
	\end{equation*}
 Constant $C$ is independent of $\epsilon$.
\end{lemma}
\begin{proof}
	The case of $1< p< +\infty$ follows directly by the property of Newtonian potential:
\begin{equation*}
	\Vert \nabla^2 w \Vert_{L^{p}(\Omega)}
	\le 
	C \Vert f^\epsilon \Vert_{L^{p}(\Omega)}
	\le 
	C \Vert f^\epsilon \Vert_{L^{p}(\Omega^\epsilon)},
\end{equation*}
and the fact
\begin{equation}\label{boundary layer estimate}
	\Vert f^\epsilon \Vert_{L^{p}(\Omega^\epsilon)}
	\le 
	\abs{\Omega^\epsilon}^{1/p} 
	\Vert f^\epsilon \Vert_{L^{\infty}(\Omega)}.
\end{equation}
Now let us consider the case of $p=1$ and write 
\begin{equation*}
	\begin{aligned}
		\frac{\partial^2 w}{\partial x_i \partial x_j} 
		= &
		\int_{\Omega} \frac{\partial^2 }{\partial x_i \partial x_j} \big\{\Phi(\bx-\boldsymbol{z})\big\} 
		\cdot \big\{ f^\epsilon(\boldsymbol{z}) 
		-f^\epsilon(\bx)
		\big\}\d \boldsymbol{z}\\
		& +
		f^\epsilon(\bx)
		\int_{\partial\Omega} \nu^i\cdot \frac{\partial }{\partial x_j} \big\{\Phi(\bx-\boldsymbol{z})\big\} \d \boldsymbol{z}\\
		=:& S_1 + S_2.
	\end{aligned}
\end{equation*}
For $S_1$, one can apply Lemma \ref{lemma: singula integral} to derive
\begin{equation*}
	\Vert S_1 \Vert_{L^{1}(\Omega)}
	\le C \epsilon
	\ln (\epsilon^{-1} + 1) 
	\big(\Vert f^\epsilon(\bx) \Vert_{L^{\infty}(\Omega)} 
	+
	\epsilon\Vert \nabla f^\epsilon(\bx) \Vert_{L^{\infty}(\Omega)}\big).
\end{equation*}
As for $S_2$, we can split the integral into $\int_{\partial\Omega-B(\bx,\epsilon)} + \int_{\partial\Omega\cap B(\bx,\epsilon)}$, and write
\begin{equation*}
	\begin{aligned}
		\Vert S_2 \Vert_{L^1(\Omega)} 
		\le &
		\sup\limits_{\bx\in\Omega}
		\int_{\partial\Omega-B(\bx,\epsilon)} \nu^i\cdot \frac{\partial }{\partial x_j} \big\{\Phi(\bx-\boldsymbol{z})\big\} \d \boldsymbol{z}
		\times 
		\int_{\Omega} f^\epsilon(\bx) \d \bx \\
		&\quad+
		\sup\limits_{\bz\in\partial\Omega}
		\int_{\Omega\cap B(\bz,\epsilon)} \nu^i\cdot \frac{\partial }{\partial x_j} \big\{\Phi(\bx-\boldsymbol{z})\big\} \cdot f^\epsilon(\bx) \d \bx
		\times 
		\int_{\partial\Omega} 1 \d \bz \\
		\le & C
		\ln (\epsilon^{-1}+1)  \times \epsilon \Vert f^\epsilon \Vert_{L^\infty(\Omega)}+C\epsilon \Vert f^\epsilon \Vert_{L^\infty(\Omega)},
	\end{aligned}
\end{equation*}
here in the second line we have used the Fubini's theorem.
Thus the Lemma is proved.
\end{proof}

\subsection{Consistency error of stray field}\label{sec:Consistency of stray field}
Now we are ready to prove the consistency error of stray field $\widetilde{\mathbf{h}}$ in Lemma \ref{consistency of stray field first part}. The idea is to use result in \cite{Praetorius2004AnalysisOT} and Green's representation formula, to rewrite $\widetilde{\mathbf{h}}$ into singular integral that of the types estimated in above Lemmas.

\begin{proof}(Proof of Lemma \ref{consistency of stray field first part})
	Recall from \eqref{define hd} the stray field in LLG equation can be calculated by
	\begin{equation}\label{stray field of Meps - Mh}
		\hd[(M^\epsilon - M^\mathrm{h}) \m_0]
		=
		\nabla U,
	\end{equation}
	where $U = U[(M^\epsilon - M^\mathrm{h}) \m_0]$ satisfies
	\begin{equation*}
		\Delta U
		=
		-\bdiv [(M^\epsilon - M^\mathrm{h}) \m_0\mathcal{X}_\Omega]
		\quad \mbox{in $D'(R^n)$}.
	\end{equation*}
	Denotes the $i$th component of $\m_0$ by $m_{0,i}$. Using the fact $\abs{\m_0} = 1$, one can write  \cite{Praetorius2004AnalysisOT}
	\begin{equation*}
		\begin{aligned}
			U(\bx)
			= &
			- \sum_{i=1}^n \int_{\Omega}\frac{\partial}{\partial x_i} \Phi(\bx-\boldsymbol{z}) (M_s(\frac{\boldsymbol{z}}{\epsilon}) - M^\mathrm{h})  m_{0,i} (\boldsymbol{z}) \d \boldsymbol{z}.
		\end{aligned}
	\end{equation*}
	Substituting above representation of $U(\bx)$ into \eqref{stray field of Meps - Mh} and making the use of cut-off function $\eta^\epsilon$ defined in \eqref{cut off function}, one can derive
	\begin{equation}\label{rewrite hd[Meps m0 - Mh m0]}
		\begin{aligned}
			&\hd[(M^\epsilon - M^\mathrm{h}) \m_0]\\
			= &-
			\nabla \Big( \sum_{i=1}^n \int_{\Omega} \frac{\partial}{\partial x_i}\Phi(\bx-\boldsymbol{z}) \eta^\epsilon(\boldsymbol{z})
		(M(\frac{\boldsymbol{z}}{\epsilon}) - M^\mathrm{h})
		 m_{0,i} (\boldsymbol{z}) \d \boldsymbol{z}\Big)\\
			&-
			\nabla \Big( \sum_{i=1}^n \int_{\Omega^\epsilon}
			\frac{\partial}{\partial x_i}\Phi(\bx-\boldsymbol{z}) (1-\eta^\epsilon(\boldsymbol{z}))
			(M(\frac{\boldsymbol{z}}{\epsilon}) - M^\mathrm{h}) 
			m_{0,i} (\boldsymbol{z}) \d \boldsymbol{z}\Big)\\
			=: & 
			\mathcal{P}^\epsilon
			+
			\widetilde{\mathcal{P}}^\epsilon,
		\end{aligned}
	\end{equation}
where $\widetilde{\mathcal{P}}^\epsilon$ is the derivative of Newtonian potential in boundary layer that can be estimated by Lemma \ref{Newtonian on the boundary layer}.
	Define $\widetilde{U}(\by)$ as the solution of
	\begin{equation}\label{define tilde U}
		\Delta \widetilde{U}(\by) = - (M_s(\by) - M^\mathrm{h}), \quad \mbox{$U(\by)$ is $Y$-periodic in $\by$},
	\end{equation}
	then one can write from \eqref{define H_d} and \eqref{define tilde U} that
\begin{equation}\label{expression of Hd}
	\begin{aligned}
	\mathbf{H}_{\mathrm{d}}[M_s(\by)](\frac{\bx}{\epsilon})
	= & 
	\epsilon^2 \nabla^2 \widetilde{U}(\frac{\bx}{\epsilon})\\
	= &
	\epsilon^2 \nabla^2 \big\{\eta^\epsilon(\bx) \widetilde{U}(\frac{\bx}{\epsilon})\big\}
	+
	\epsilon^2 \nabla^2 \big\{(1-\eta^\epsilon(\bx)) \widetilde{U}(\frac{\bx}{\epsilon})\big\}.
	\end{aligned}
\end{equation}
Note that by Green's representation formula,
\begin{equation*}\label{Green representation}
	\begin{aligned}
		\epsilon^2 \eta^\epsilon(\bx) \widetilde{U}(\frac{\bx}{\epsilon})
		= & -
		\int_{\Omega} \Phi(\bx-\bz) 
		\Delta \big(\epsilon^2\eta^\epsilon (\bz)\widetilde{U}(\frac{\bz}{\epsilon})\big) \d \bz.
	\end{aligned}
\end{equation*}
Substituting the above formula into \eqref{expression of Hd} and using the fact of $\widetilde{U}$
\begin{equation*}
	- \Delta \big(\epsilon^2\eta^\epsilon (\bz)\widetilde{U}(\frac{\bz}{\epsilon})\big)
	=
	\eta^\epsilon (M(\frac{\boldsymbol{z}}{\epsilon}) - M^\mathrm{h}) 
	- 
	\big\{ \epsilon^2 \Delta\eta^\epsilon (\bz) \cdot \widetilde{U}(\frac{\bz}{\epsilon})
	+ 2 \epsilon^2 \nabla\eta^\epsilon (\bz) \cdot \nabla\widetilde{U}(\frac{\bz}{\epsilon}) \big\},
\end{equation*}
we finally obtain
\begin{equation*}\label{final value of H m0}
	\begin{aligned}
		&\m_0 \cdot \mathbf{H}_{\mathrm{d}}[M_s(\by)](\frac{\bx}{\epsilon})\\
		= &
		\m_0 \cdot \nabla^2 \int_{\Omega} \Phi(\bx-\bz) 
		\eta^\epsilon (\bz)
		(M(\frac{\boldsymbol{z}}{\epsilon}) - M^\mathrm{h}) \d \bz
		 + \m_0\cdot  \Big\{ 
		\epsilon^2 \nabla^2 \big\{(1-\eta^\epsilon(\bx)) \widetilde{U}(\frac{\bx}{\epsilon})\big\}\\
		& +
		\nabla^2 \int_{\Omega} \Phi(\bx-\bz) 
		\big\{ \epsilon^2 \Delta\eta^\epsilon (\bz) \cdot \widetilde{U}(\frac{\bz}{\epsilon})
		+ 2 \epsilon^2 \nabla\eta^\epsilon (\bz) \cdot \nabla\widetilde{U}(\frac{\bz}{\epsilon}) \big\} \d \bz
		 \Big\}\\
		=: & \mathcal{Q}^\epsilon
		+
		\widetilde{\mathcal{Q}}^\epsilon,
	\end{aligned}
\end{equation*}
where the boundary layer term $\widetilde{\mathcal{Q}}^\epsilon$ can be estimated by Lemma \ref{Newtonian on the boundary layer} and \eqref{boundary layer estimate}.
Now in order to estimate the left-hand side of \eqref{estimate of stray field} in the Lemma, it only remains to consider the term $\mathcal{P}^\epsilon - \mathcal{Q}^\epsilon$. Notice that one can write
\begin{equation}\label{P-Q}
	\begin{aligned}
	\mathcal{P}^\epsilon - \mathcal{Q}^\epsilon
	=
	\sum_{i=1}^n \int_{\Omega} &\frac{\partial}{\partial x_i} 
	\big\{\nabla_{\bx} \Phi(\bx-\boldsymbol{z})\big\} \eta^\epsilon(\boldsymbol{z})\\
	&\times (M(\frac{\boldsymbol{z}}{\epsilon}) - M^\mathrm{h})
	\big(m_{0,i} (\bx) - m_{0,i} (\boldsymbol{z})\big) \d \boldsymbol{z}.
	\end{aligned}
\end{equation}
With the notation \eqref{define tilde U}, one has
	\begin{equation*}
	(M(\frac{\boldsymbol{z}}{\epsilon}) - M^\mathrm{h})
	=
	\nabla_{\bz} \cdot\big\{ 
	\nabla_{\bz} (\epsilon^2 \widetilde{U}(\frac{\boldsymbol{z}}{\epsilon}))
	-
	\nabla_{\bx} (\epsilon^2 \widetilde{U}(\frac{\bx}{\epsilon})) \big\}.
\end{equation*}
After substituting it into \eqref{P-Q} and applying integration by parts, the leading integrals are estimated directly by application of Lemma \ref{lemma: singula integral}.
\end{proof}

\section{Boundary Corrector}\label{section:boundary corrector}

\subsection{Neumann corrector}
Let us give the definition of Neumann corrector $\boldsymbol{\omega}_\mathrm{N}$ as
\begin{equation}\label{define w_b}
	\boldsymbol{\omega}_\mathrm{N}
	=
	\sum_{i = 1}^n \big(\Phi_i - x_i - \epsilon\chi^\epsilon_i \big)
	\frac{\partial \m^0}{\partial x_i}
\end{equation}
with the notation $\chi^\epsilon_i(\bx) = \chi_i(\frac{\bx}{\epsilon})$, $x_i$ is the $i$th component of spatial variable, and $(\Phi_i)_{1\le i \le n}$ is given by
\begin{equation}\label{define Phi}
\left\{\begin{aligned}
\bdiv (a^\epsilon \nabla \Phi_i) &= \bdiv (a^0 \nabla x_i) \quad \mbox{in $\Omega$} ,\\
\frac{\partial}{\partial \boldsymbol{\nu}^\epsilon} \Phi_i &= \frac{\partial}{\partial \boldsymbol{\nu}^\mathrm{h}} x_i \quad \mbox{on $\partial\Omega$} .
\end{aligned}\right.
\end{equation}
Here we denote $\frac{\partial}{\partial \boldsymbol{\nu}^\epsilon} = \boldsymbol{\nu}\cdot a^\epsilon\nabla$, $\frac{\partial}{\partial \boldsymbol{\nu}^\mathrm{h}} = \boldsymbol{\nu} \cdot a^0 \nabla$. Thus $x_i$ is the homogenized solution of $\Phi_i$ from \eqref{define Phi}. Since $\Phi_i$ is unique up to a constant, one may assume $\Phi_i (\tilde{\bx}) - \tilde{\bx}= 0$  for some $\tilde{\bx}\in \Omega$. We introduce that $\Phi_i - x_i - \epsilon\chi^\epsilon_i$ has following property.
\begin{lemma}\label{lemma: estimate of chi}
	For $\Phi_i$ given in \eqref{define Phi}, under the smoothness assumption on $A(y)$ and $\partial\Omega$, it holds that (see \cite{shen2018periodic})
	\begin{equation}\label{L2 ound of nabla wb}
	\Vert \nabla \Phi_i - \nabla x_i - \epsilon\nabla\chi^\epsilon_i \Vert_{L^{\infty}(\Omega)} \le C,
	\quad
	\Vert \nabla^2 \Phi_i \Vert_{L^{\infty}(\Omega)} \le C,
	\end{equation}
	and
	\begin{equation}\label{L^infty bound of Phi-x}
	\Vert \Phi_i - x_i \Vert_{L^\infty(\Omega)} \le C \epsilon \ln (\epsilon^{-1} + 1),
	\end{equation}
	where $C$ is independent of $\epsilon$.
\end{lemma}
\begin{proof}
	In fact, one has the estimate
	\begin{equation*}
	\big\vert\nabla \Phi_i - \nabla x_i - \epsilon\nabla\chi^\epsilon_i\big\vert
	\le C \max\{1,\epsilon [\mathrm{dist}(\bx,\partial\Omega)]^{-1}\}
	\end{equation*}
	from Lemma $7.4.5$ in \cite{shen2018periodic}. This, together with the fact $\Phi_i (\tilde{\bx}) - \tilde{\bx}= 0 $ , yields \eqref{L^infty bound of Phi-x} by following integrals:
	\begin{equation*}
		\begin{aligned}
		\abs{\Phi_i (\bx) - x_i} 
		=&
		\Big\vert \int_0^1 \frac{\d}{\d s} \Big\{\Phi_i \big(\tilde{\bx} + s(\bx-\tilde{\bx})\big) - \big(\tilde{x}_i + s (x_i - \tilde{x}_i)\big) \Big\} \d s \Big\vert\\
		\le & C
		\int_0^1 \max\{1,\epsilon (1-s)^{-1}\} \d s 
		\le  C \epsilon \ln (\epsilon^{-1} + 1),
		\end{aligned}
\end{equation*}
for any $\bx \in \Omega$.

As for the second inequality in \eqref{L2 ound of nabla wb}, we prove by making use of the Neumann function for operator $\mathcal{A}_\epsilon$ from \cite{shen2018periodic} Section 7.4, denoted by $N^\epsilon(\bx,\bz)$, and write from \eqref{define Phi} that
\begin{equation}\label{representation of Phi}
	\begin{aligned}
	\Phi_i(\bx)
	=&- \sum_{k=1}^{n}
	\int_{\partial\Omega} \nu_k \cdot a^0_{ki} N^\epsilon(\bx,\bz) \d\bz
	+ \frac{1}{\abs{\partial\Omega}}\int_{\partial\Omega}\Phi_i(\bz) \d \bz.
	\end{aligned}
\end{equation}
Let us denote the projection of $\frac{\partial}{\partial x_j}$ along $\frac{\partial}{\partial \boldsymbol{\nu}^\epsilon}$ by $\mathcal{P}_{x_j}$, and define $\mathcal{P}_{x_j}^{\bot} = \frac{\partial}{\partial \boldsymbol{\nu}^\epsilon} - \mathcal{P}_{x_j}$, one can write for $\bz\in \partial\Omega$
\begin{equation*}
	\frac{\partial}{\partial z_j} N^\epsilon(\bx,\bz)
	=
	\big( \mathcal{P}_{z_j} + \mathcal{P}_{z_j}^{\bot} \big)
	N^\epsilon(\bx,\bz)
	=
	\mathcal{P}_{z_j}^{\bot} 
	N^\epsilon(\bx,\bz).
\end{equation*}
Now applying $\frac{\partial^2}{\partial x_l\partial x_j}$
to both sides of \eqref{representation of Phi}, using above formula and integration by parts on $\partial\Omega$ lead to
\begin{equation}\label{representation of second derivative of Phi}
	\begin{aligned}
		\frac{\partial^2}{\partial x_l\partial x_j}\Phi_i(\bx)
		=&- \sum_{k=1}^{n}
		\int_{\partial\Omega} \mathcal{P}_{z_l}^{\bot}\mathcal{P}_{z_j}^{\bot}\nu_k(\bz) \cdot a^0_{ki} N^\epsilon(\bx,\bz) \d\bz,
	\end{aligned}
\end{equation}
here we have used the fact that $\mathcal{P}_{z_l}^{\bot}$ is a tangential derivative on the boundary, and $N^\epsilon(\bx,\bz) = N^\epsilon(\bz,\bx)$ by the symmetry of $\mathcal{A}_\epsilon$. 
\eqref{representation of second derivative of Phi} implies the second inequality in \eqref{L2 ound of nabla wb} by the smoothness assumption of boundary.
\end{proof}

\subsection{A high-order modification}
As noted in Section \ref{sec:introduction}, we use $\boldsymbol{\omega}_\mathrm{N}$ to control the Neumann boundary data, and use a modification function $\boldsymbol{\omega}_\mathrm{M}$ to control the inhomogeneous term that induced by $\boldsymbol{\omega}_\mathrm{N}$, written in \eqref{terms of boundary corrector} separately.
In order to explain the construction of the modification function, we point out that there are some bad terms appear when we calculate $\mathcal{L}^\epsilon \boldsymbol{\omega}_\mathrm{N}$, which have no convergence in $L^2$ norm. Denote the bad terms by $\mathcal{T}_{\mathrm{bad}}^1$ and $\mathcal{T}_{\mathrm{bad}}^2$, then they can be written as
	\begin{equation}\label{bad term 1}
		\begin{aligned}
			\mathcal{T}_{\mathrm{bad}}^1
			=&
			2 \sum_{i, j, k = 1}^n \frac{\partial}{\partial x_k} \Big\{a^\epsilon_{ki} 
			\big(\Phi_j - x_j - \epsilon\chi^\epsilon_j \big) \cdot
			\frac{\partial^2 \m_0}{\partial x_i \partial x_j} \Big\}\\
			& - 
			\sum_{i, j, k = 1}^n \Big\{ \frac{\partial}{\partial x_k} a^\epsilon_{ik} \cdot
			\big(\Phi_j - x_j - \epsilon\chi^\epsilon_j \big) \Big\}
			\frac{\partial^2 \m_0}{\partial x_i \partial x_j},
		\end{aligned}
	\end{equation}
and
\begin{equation*}
	\begin{gathered}
		\mathcal{T}_{\mathrm{bad}}^2
		=
		\alpha \sum_{i, j= 1}^n
		\big(a^\epsilon_{ij}
		\frac{\partial\boldsymbol{\omega}_\mathrm{N}}{\partial x_i}  
		\cdot \frac{\partial\{2\widetilde{\m}^\epsilon + \boldsymbol{\omega}_\mathrm{N}\}}{\partial x_j}  
		\big)\big(\widetilde{\m}^\epsilon + \boldsymbol{\omega}_\mathrm{N}\big).
	\end{gathered}
\end{equation*}
Notice that these terms cannot converge for the existence of $\frac{\partial\boldsymbol{\omega}_\mathrm{N}}{\partial x_i} $ and $\frac{\partial a^\epsilon_{ik}}{\partial x_k}$. 
Now let us rewrite $\mathcal{T}_{\mathrm{bad}}^1$ and $\mathcal{T}_{\mathrm{bad}}^2$ into divergence form up to a small term. For $\mathcal{T}_{\mathrm{bad}}^1$, notice that $\sum_{k = 1}^n \frac{\partial a^\epsilon_{ik}}{\partial x_k} = \mathcal{A}_\epsilon (\epsilon\chi^\epsilon_i)$ from \eqref{eqn:chi_j}, substitute it into the second term on the right-hand side of \eqref{bad term 1}, it leads to
\begin{equation}\label{difference of T_bad^1}
	\begin{aligned}
		\mathcal{T}_{\mathrm{bad}}^1
		=&
		\sum_{k, l = 1}^n \frac{\partial}{\partial x_k} \big(a^\epsilon_{kl} G^1_l(\bx)\big)\\
		& + 
		\sum_{i, j = 1}^n 
		\epsilon\chi^\epsilon_i \cdot
		\mathcal{A}_\epsilon \Big\{
		\big(\Phi_j - x_j - \epsilon\chi^\epsilon_j \big) \cdot
		\frac{\partial^2 \m_0}{\partial x_i \partial x_j}\Big\},
	\end{aligned}
\end{equation}
where $G^1_l(\bx)$ in the divergence term reads
\begin{align*}
	G^1_l(\bx)=
	&
	2 \sum_{ j = 1}^n  
	\big(\Phi_j - x_j - \epsilon\chi^\epsilon_j \big)
	\frac{\partial^2 \m_0}{\partial x_l \partial x_j}
	+ 
	\sum_{i, j = 1}^n
	\Big\{\frac{\partial}{\partial x_l} 
	( \epsilon\chi^\epsilon_i) \nn\\
	&\cdot
	\big(\Phi_j - x_j - \epsilon\chi^\epsilon_j \big) 
	- 
	 \epsilon\chi^\epsilon_i \cdot
	\frac{\partial}{\partial x_l}
	\big(\Phi_j - x_j - \epsilon\chi^\epsilon_j \big) \Big\}
	\frac{\partial^2 \m_0}{\partial x_i \partial x_j}.
\end{align*}
As for $\mathcal{T}_{\mathrm{bad}}^2$, a direct calculation implies it can be rewritten as
	\begin{align}\nn
	\mathcal{T}_{\mathrm{bad}}^2
	=&\sum_{k, l = 1}^n \frac{\partial}{\partial x_k} \big(a^\epsilon_{kl} G^2_l(\bx)\big)
	-
	\alpha \sum_{i, j= 1}^n
	a^\epsilon_{ij} \big(
	\boldsymbol{\omega}_\mathrm{N}
	\cdot \frac{\partial\{2\widetilde{\m}^\epsilon + \boldsymbol{\omega}_\mathrm{N}\}}{\partial x_j}  
	\big)
	\big(\widetilde{\m}^\epsilon + \boldsymbol{\omega}_\mathrm{N}\big)\\
	\label{difference of T_bad^2}
	& -
	\alpha \sum_{i, j= 1}^n
	\big(
	\boldsymbol{\omega}_\mathrm{N}
	\cdot \mathcal{A}_\epsilon\{2\widetilde{\m}^\epsilon + \boldsymbol{\omega}_\mathrm{N}\}
	\big)
	\big(\widetilde{\m}^\epsilon + \boldsymbol{\omega}_\mathrm{N}\big),
\end{align}
where $G^2_l(\bx)$ in the divergence term can be calculated by
\begin{equation*}
	\begin{aligned}
		G^2_l(\bx)=
		& \alpha
		\big(
		\boldsymbol{\omega}_\mathrm{N} 
		\cdot \frac{\partial\{2\widetilde{\m}^\epsilon + \boldsymbol{\omega}_\mathrm{N}\}}{\partial x_l}  
		\big)\big(\widetilde{\m}^\epsilon + \boldsymbol{\omega}_\mathrm{N}\big).
	\end{aligned}
\end{equation*}
Moreover, one can apply Lemma \ref{lemma: estimate of chi} to deduce from \eqref{difference of T_bad^1} and \eqref{difference of T_bad^2} that for $i=1,2$
\begin{equation}\label{difference of bad term}
	\begin{aligned}
		\big\Vert \mathcal{T}_{\mathrm{bad}}^i
		- \sum_{k, l = 1}^n \frac{\partial}{\partial x_k} \big(a^\epsilon_{kl} G^i_l(\bx)\big)
		 \big\Vert_{L^2(\Omega)}
		\le&
		C \epsilon [\ln (\epsilon^{-1} + 1)]^2,
	\end{aligned}
\end{equation}
here we have use the fact $\mathcal{A}_\epsilon (\Phi_i - x_i - \epsilon\chi^\epsilon_i ) = 0$. Constant $C$ depends on $\Vert \nabla \m_0 \Vert_{W^{2,\infty}(\Omega)}$, $\Vert \mathcal{A}_\epsilon \widetilde{\m}^\epsilon \Vert_{L^{\infty}(\Omega)}$, but is independent of $\epsilon$.

Now we define the modification function $\boldsymbol{\omega}_\mathrm{M} = \boldsymbol{\omega}_\mathrm{M}^1 + \boldsymbol{\omega}_\mathrm{M}^2$,
where $\boldsymbol{\omega}_\mathrm{M}^i$, $i=1,2$ satisfies
\begin{equation}\label{define tilde omega}
\left\{ \begin{aligned}
\mathcal{A}_\epsilon \boldsymbol{\omega}_\mathrm{M}^i
=&
\sum_{k, l = 1}^n \frac{\partial}{\partial x_k} \big(a^\epsilon_{kl} G^i_l(\bx)\big)
\quad \mbox{in $\Omega$},\\
\frac{\partial}{\partial \boldsymbol{\nu}^\epsilon} \boldsymbol{\omega}_\mathrm{M}^i
=&
\sum_{k, l = 1}^n
\nu^k \cdot a^\epsilon_{kl} 
G^i_l(\bx)
\quad \mbox{on $\partial\Omega$},
\end{aligned} \right.
\end{equation}
here $\nu^k$ is the $k$-th component of vector $\boldsymbol{\nu}$. By the Lax-Milgram theorem, one can obtain the existence and uniqueness of $\boldsymbol{\omega}_\mathrm{M}^i$, $i=1,2$ up to a constant. Let $\int_{\partial\Omega} \boldsymbol{\omega}_\mathrm{M}^i \d \bx = 0$, then the correctors yield the following estimate.
\begin{lemma}\label{lemma: estimate of theta and gamma}
	For $\boldsymbol{\omega}_\mathrm{M}^i$, $i=1,2$ defined in \eqref{define tilde omega}, under smooth assumption of $\m_0$ and $\partial \Omega$, it holds that for $n\le 3$
	\begin{equation}\label{estimate of omega}
	\Vert \boldsymbol{\omega}_\mathrm{M}^i \Vert_{L^\infty(\Omega)}
	\le C \epsilon, 
	\quad 
	\Vert \nabla \boldsymbol{\omega}_\mathrm{M}^i \Vert_{L^\infty(\Omega)}
	\le
	C \epsilon \ln (\epsilon^{-1} + 1),
	\end{equation}
	where $C$ depends on $\Vert \nabla \m_0 \Vert_{W^{3,\infty}(\Omega)}$ and is independent of $\epsilon$.
\end{lemma}
\begin{proof}
	Here we use the Neumann function $N^\epsilon(\bx,\bz)$ for operator $\mathcal{A}_\epsilon$, see \cite{shen2018periodic} Section 7.4. \eqref{define tilde omega} implies for $i=1,2$
	\begin{equation*}
	\boldsymbol{\omega}_\mathrm{M}^i 
	= \sum_{k, l = 1}^n  
	\int_{\Omega} a^\epsilon_{kl} \frac{\partial}{\partial z_k}\big\{
	N^\epsilon(\bx,\bz) \big\} G^i_l(\bz) \d \bz.
	\end{equation*}
	Using the fact $\nabla_{\bz} N^\epsilon(\bx,\bz)\le C\abs{\bx-\bz}^{1-n}$, see  \cite{shen2018periodic} p.159, we can derive
	\begin{equation*}
	\Vert \boldsymbol{\omega}_\mathrm{M}^i  \Vert_{L^\infty(\Omega)} 
	\le C\Vert G^i_l \Vert_{L^\infty(\Omega)}
	\le 
	C\epsilon \ln (\epsilon^{-1} + 1).
	\end{equation*}
	As for the second inequality in \eqref{estimate of omega}, it follows from \cite{shen2018periodic}, Lemma 7.4.7:
	\begin{equation*}
	\Vert \nabla \boldsymbol{\omega}_\mathrm{M}^i  \Vert_{L^\infty(\Omega)} \le 
	C\ln (\epsilon^{-1} + 1)\Vert G^i_l \Vert_{L^\infty(\Omega)}
	+
	C\epsilon \Vert \nabla G^i_l \Vert_{L^\infty(\Omega)}
	\end{equation*}
	with the estimate 
	\begin{equation*}
		\Vert \nabla G^i_l \Vert_{L^\infty(\Omega)} \le C,
	\end{equation*}
	from Lemma \ref{lemma: estimate of chi}.
	Here constant $C$ depends on $\Vert \nabla \m_0 \Vert_{W^{3,\infty}(\Omega)}$, $\Vert \nabla (\Phi_j - x_j - \epsilon\chi^\epsilon_j) \Vert_{L^\infty(\Omega)}$, but is independent of $\epsilon$ by Lemma \ref{lemma: estimate of chi}.
\end{proof}

\subsection{Estimates of initial-boundary data}
\begin{theorem}\label{thm: Estimates of initial-boundary data}
	For $\e^\epsilon_{{\mathrm{b}}}$ given in \eqref{define e_b in intro}, with $\boldsymbol{\omega}_\mathrm{b} = \boldsymbol{\omega}_\mathrm{N}
	-\boldsymbol{\omega}_\mathrm{M}$ given in \eqref{define w_b}, under the smooth assumption, it holds that initial data of $\e^\epsilon_{{\mathrm{b}}}$ satisfies
	\begin{equation}\label{initial condition}
		\Vert \e^\epsilon_{{\mathrm{b}}}(\bx,0) \Vert_{H^{1}(\Omega)}\le C\epsilon \ln (\epsilon^{-1} + 1),
	\end{equation}
	where $C$ depends on $\Vert \nabla^2 \m_{\mathrm{init}}^0 \Vert_{H^{1}(\Omega)}$ and is independent of $\epsilon$.
	And for the boundary data, it holds that
\begin{align}\label{boundary estimate of e}
	&\Vert \frac{\partial}
	{\partial \boldsymbol{\nu}^\epsilon} \e^\epsilon_{{\mathrm{b}}} \Vert_{W^{1,\infty}(0,T; B^{-1/2, 2} (\partial\Omega))}
	\le C \epsilon \ln (\epsilon^{-1} + 1),
\end{align}
where $C$ depends on $\Vert \nabla^2 \m_0 \Vert_{W^{1,\infty}(0,T; B^{-1/2, 2} (\partial\Omega))}$ and is independent of $\epsilon$.
\end{theorem}

\begin{proof}
We rewrite $\e^\epsilon_{{\mathrm{b}}}$ from its definition as
\begin{equation}\label{rewrite e_b}
	\e^\epsilon_{{\mathrm{b}}} 
	= 
	\m^\epsilon - \m_0 - \sum_{i=1}^n(\Phi_i - x_i )\frac{\partial\m_0}{\partial x_i} - \epsilon^2\m_2 + \boldsymbol{\omega}_\mathrm{M}.
\end{equation}
First, let us prove \eqref{initial condition}. By the initial condition of $\m^\epsilon$ and $\m_0$, along with the smoothness condition, one can check
\begin{equation}\label{initial data of eb}
	\begin{aligned}
		\e^\epsilon_{{\mathrm{b}}}(\bx,0)
		= 
		\m_{\mathrm{init}}^\epsilon - \m_{\mathrm{init}}^0 - \sum_{i=1}^n(\Phi_i - x_i )\frac{\partial\m_{\mathrm{init}}^0}{\partial x_i} 
		-
		\epsilon^2\m_{2,\mathrm{init}}
		+ \boldsymbol{\omega}_\mathrm{M,init}
	\end{aligned}
\end{equation}
with notation $\m_{2,\mathrm{init}}$, $\boldsymbol{\omega}_\mathrm{M,init}$ defined the same as $\m_{2}$, $\boldsymbol{\omega}_\mathrm{M}$ except we replace $\m_0$ by $\m_{\mathrm{init}}^0$.
From the assumption \eqref{compatibility condition of initial data}-\eqref{initial data}, $\m_{\mathrm{init}}^0$ is the homogenized solution of $\m_{\mathrm{init}}^\epsilon$, by classical homogenization theorem of elliptic problems in \cite{shen2018periodic}, one has
\begin{equation*}
	\Vert \m_{\mathrm{init}}^\epsilon - \m_{\mathrm{init}}^0 - (\boldsymbol{\Phi} - \boldsymbol{x} )\nabla\m_{\mathrm{init}}^0 \Vert_{H^{1}(\Omega)}\le C\epsilon\ln (\epsilon^{-1} + 1).
\end{equation*}
Also note that by definition of $\m_2$ and Lemma \ref{lemma: estimate of theta and gamma}, one has
\begin{equation*}
	\Vert\epsilon^2 \m_{2,\mathrm{init}} \Vert_{H^{1}(\Omega)}
	+
	\Vert \boldsymbol{\omega}_\mathrm{M,init} \Vert_{H^{1}(\Omega)}
	\le C \epsilon\ln (\epsilon^{-1} + 1) \Vert \nabla^2 \m_{\mathrm{init}}^0 \Vert_{H^{1}(\Omega)}.
\end{equation*}
Therefore the inequality \eqref{initial condition} follows from \eqref{initial data of eb} and above estimates.

Notice that by the boundary condition of $\m_0$ in \eqref{eqn:homogenized LLG system} and $\Phi_k$ in \eqref{define Phi}, we have 
\begin{equation*}
	\begin{aligned}
		\sum_{k=1}^n \frac{\partial}
		{\partial \boldsymbol{\nu}^\epsilon}(\Phi_k - x_k ) \cdot \frac{\partial\m_0}{\partial x_k}
		=&
		- \frac{\partial}
		{\partial \boldsymbol{\nu}^\epsilon} \m_0 ,
		\quad \bx\in \partial\Omega,
	\end{aligned}
\end{equation*}
therefore applying $\frac{\partial}{\partial \boldsymbol{\nu}^\epsilon}$ to both sides of \eqref{rewrite e_b} leads to
\begin{equation}\label{boundary of e_b}
	\begin{aligned}
		\frac{\partial}
		{\partial \boldsymbol{\nu}^\epsilon}  \e^\epsilon_{{\mathrm{b}}}
		=&
		- \sum_{k=1}^n (\Phi_k - x_k ) \cdot 
		\frac{\partial}
		{\partial \boldsymbol{\nu}^\epsilon} (\frac{\partial\m_0}{\partial x_k})
		- \epsilon^2
		\frac{\partial}
		{\partial \boldsymbol{\nu}^\epsilon} \m_2
		+
		\frac{\partial}
		{\partial \boldsymbol{\nu}^\epsilon} \boldsymbol{\omega}_\mathrm{M} ,
		\quad \bx\in \partial\Omega.
	\end{aligned}
\end{equation}
Under the smoothness assumption of $\m_0$ and $a^\epsilon$, we can also derive the smoothness of $(\boldsymbol{\Phi} - \boldsymbol{x})$, $\m_2$ and $\boldsymbol{\omega}_\mathrm{M}$ over $\bar{\Omega}$.
Thus by Lemma \ref{lemma: estimate of chi} and Lemma \ref{lemma: estimate of theta and gamma}, one can directly obtain from \eqref{boundary of e_b}
\begin{align*}
	\Vert \frac{\partial}
	{\partial \boldsymbol{\nu}^\epsilon} \e^\epsilon_{{\mathrm{b}}} \Vert_{ B^{-1/2, 2} (\partial\Omega)}
	\le &
	C \Vert \boldsymbol{\Phi} - \boldsymbol{x} \Vert_{L^\infty(\Omega)}
	+
	C\epsilon^2 \Vert \nabla \m_2 \Vert_{L^{\infty} (\Omega)}
	+ 
	C\Vert \nabla \boldsymbol{\omega}_\mathrm{M} \Vert_{L^{\infty} (\Omega)}\\
	\le &C \epsilon \ln (\epsilon^{-1} + 1),
\end{align*}
where $C$ depends on $\Vert \nabla^2 \m_0 \Vert_{B^{-1/2, 2} (\partial\Omega)}$ and $\Vert \nabla^2 (\partial_t \m_0) \Vert_{B^{-1/2, 2} (\partial\Omega)}$. The same argument for $\Vert \frac{\partial}
{\partial \boldsymbol{\nu}^\epsilon} (\partial_t \e^\epsilon_{{\mathrm{b}}}) \Vert_{ B^{-1/2, 2} (\partial\Omega)}$
leads to \eqref{boundary estimate of e}.


\end{proof}

\subsection{Estimates of inhomogeneous terms}
From the above definition and property, we get the main result of this section.
\begin{theorem}\label{theorem:omega}
	For $\e^\epsilon_{{\mathrm{b}}}$ given in \eqref{define e_b in intro}, with $\boldsymbol{\omega}_\mathrm{b} = \boldsymbol{\omega}_\mathrm{N}
	-\boldsymbol{\omega}_\mathrm{M}$ given in \eqref{define w_b} and $\mathcal{L}^\epsilon$ defined in \eqref{def L eps}, under the smooth assumption, it holds that
	\begin{gather}
		\label{estimate of partial_t w_b}
		\Vert \partial_t \boldsymbol{\omega}_{{\mathrm{b}}} \Vert_{L^{2}(\Omega)}
		\le C \epsilon \ln (\epsilon^{-1} + 1),\\
	\label{estimate of L w_b}
	\Vert \mathcal{L}^\epsilon\boldsymbol{\omega}_\mathrm{b}
	\Vert_{L^2(\Omega)}
	\le 
	C \epsilon [\ln (\epsilon^{-1} + 1)]^2
	+
	C\Vert \m^\epsilon - \widetilde{\m}^\epsilon - \boldsymbol{\omega}_\mathrm{b} \Vert_{H^{1}(\Omega)},
	\end{gather}
where $C$ depends on $
\Vert \m^\epsilon \Vert_{H^{1}(\Omega)}$, $\Vert \nabla^2 \m_0  \Vert_{W^{2,\infty}(\Omega)}$, $\Vert \nabla (\partial_t \m_0) \Vert_{W^{1,\infty}(\Omega)}$ and is independent of $\epsilon$.
	Moreover, one has the estimate
	\begin{equation}\label{estimate of A e_b}
	\Vert \mathcal{A}_\epsilon \e^\epsilon_{{\mathrm{b}}} \Vert_{L^{2}(\Omega)}
	\le C \ln (\epsilon^{-1} + 1),
	\end{equation}
	where $C$ depends on $
	\Vert \mathcal{A}_\epsilon \m^\epsilon \Vert_{L^{2}(\Omega)}$, $\Vert \nabla^2 \m_0  \Vert_{W^{2,\infty}(\Omega)}$, $\Vert \nabla (\partial_t \m_0) \Vert_{W^{1,\infty}(\Omega)}$ and is independent of $\epsilon$.
\end{theorem}
\begin{proof}
	In order to estimate left-hand side of \eqref{estimate of L w_b}, we split it as $\Vert \mathcal{L}^\epsilon\boldsymbol{\omega}_\mathrm{b}
	\Vert_{L^2(\Omega)} = \Vert \mathcal{L}^\epsilon
	\boldsymbol{\omega}_\mathrm{N} - \mathcal{L}^\epsilon \boldsymbol{\omega}_\mathrm{M} \Vert_{L^2(\Omega)} 
	\le R_1 + R_2 + R_3 $, with 
	\begin{equation*}
		\left\{
		\begin{aligned}
		R_1 =& \big\Vert \mathcal{L}^\epsilon
		\boldsymbol{\omega}_\mathrm{N} 
		-
		\big\{ \mathcal{A}_\epsilon \boldsymbol{\omega}_\mathrm{N}
		- 
		\m^\epsilon \times \mathcal{A}_\epsilon \boldsymbol{\omega}_\mathrm{N}
		-  
		\mathbf{D}_2 (\boldsymbol{\omega}_\mathrm{N})
		 \big\} \big\Vert_{L^2(\Omega)},\\
		R_2 =  &
		\big\Vert \big\{ \mathcal{A}_\epsilon \boldsymbol{\omega}_\mathrm{M}^1
		- 
		\m^\epsilon \times \mathcal{A}_\epsilon \boldsymbol{\omega}_\mathrm{M}^1
		-  
		\mathcal{A}_\epsilon \boldsymbol{\omega}_\mathrm{M}^2
		\big\}
		 -
		\mathcal{L}^\epsilon \boldsymbol{\omega}_\mathrm{M} 
		\big\Vert_{L^2(\Omega)},\\
		R_3 = &
		\big\Vert
		\mathcal{A}_\epsilon 
		(\boldsymbol{\omega}_\mathrm{N} 
		- \boldsymbol{\omega}_\mathrm{M}^1)
		- 
		\m^\epsilon \times \mathcal{A}_\epsilon 
		(\boldsymbol{\omega}_\mathrm{N} 
		- \boldsymbol{\omega}_\mathrm{M}^1)
		-  
		\big( \mathbf{D}_2 (\boldsymbol{\omega}_\mathrm{N})
		-  
		\mathcal{A}_\epsilon \boldsymbol{\omega}_\mathrm{M}^2
		\big)
		\big\Vert_{L^2(\Omega)}.
		\end{aligned}\right.
	\end{equation*}
One can check by definition of $\mathcal{L}^\epsilon$ that $R_1$ does not have derivative of $\boldsymbol{\omega}_\mathrm{N}$, and $R_2$ only contains first-order derivative of $\boldsymbol{\omega}_\mathrm{M}$, thus they can be estimated by Lemma \ref{lemma: estimate of chi} and Lemma \ref{lemma: estimate of theta and gamma} as
\begin{equation*}
	R_1 + R_2 \le C \epsilon \ln (\epsilon^{-1} + 1),
\end{equation*}
where $C$ depends on $
\Vert \m^\epsilon \Vert_{H^{1}(\Omega)}$, $\Vert \nabla^2 \m_0  \Vert_{W^{2,\infty}(\Omega)}$.
As for $R_3$, in the view of \eqref{difference of bad term}, it can be bounded from above by
\begin{equation*}
	\begin{aligned}
		\Vert \mathcal{A}_\epsilon \boldsymbol{\omega}_\mathrm{N}
		- \mathcal{T}_{\mathrm{bad}}^1 \Vert_{L^2(\Omega)}
		+
		\Vert \m^\epsilon \times (\mathcal{A}_\epsilon \boldsymbol{\omega}_\mathrm{N}
		- \mathcal{T}_{\mathrm{bad}}^1) \Vert_{L^2(\Omega)}
		+
		\Vert \mathbf{D}_2 (\boldsymbol{\omega}_\mathrm{N})
		- \mathcal{T}_{\mathrm{bad}}^2 \Vert_{L^2(\Omega)}.
	\end{aligned}
\end{equation*}
	In above terms, the first term can be estimated by applying $\mathcal{A}_\epsilon (\Phi_i - x_i - \epsilon\chi^\epsilon_i ) = 0$ to derive that $\Vert \mathcal{A}_\epsilon \boldsymbol{\omega}_\mathrm{N}
	- \mathcal{T}_{\mathrm{bad}}^1 \Vert_{L^2(\Omega)} \le C \epsilon \ln (\epsilon^{-1} + 1)$, with $C$ independent of $\epsilon$. The same result holds for the second term. Now let us estimate the last term. We assert that
	\begin{equation}\label{D_2 - A}
	\Vert \mathbf{D}_2 (\boldsymbol{\omega}_\mathrm{N})
	- \mathcal{T}_{\mathrm{bad}}^2\Vert_{L^2(\Omega)}
	\le 
	C \epsilon\ln (\epsilon^{-1} + 1) 
	+
	C\Vert \m^\epsilon - \widetilde{\m}^\epsilon - \boldsymbol{\omega}_\mathrm{b}\Vert_{H^{1}(\Omega)},
	\end{equation}
	where $C$ depends on $\Vert \nabla^2 \m_0 \Vert_{W^{1,\infty}(\Omega)}$ and $
	\Vert \m^\epsilon \Vert_{H^{1}(\Omega)}$. In fact, we denote the terms in $\mathbf{D}_2(\boldsymbol{\omega}_\mathrm{N})$ that contain derivatives of $\boldsymbol{\omega}_\mathrm{N}$ by $\widetilde{\mathbf{D}}_2(\boldsymbol{\omega}_\mathrm{N}) $,
	then it reads
	\begin{equation*}
	\begin{gathered}
	\widetilde{\mathbf{D}}_2(\boldsymbol{\omega}_\mathrm{N}) 
	=
	\alpha \sum_{i, j= 1}^n
	\big(a^\epsilon_{ij}
	\frac{\partial\boldsymbol{\omega}_\mathrm{N}}{\partial x_i}  
	\cdot \frac{\partial\m^\epsilon}{\partial x_j}  
	+
	a^\epsilon_{ij}
	\frac{\partial\widetilde{\m}^\epsilon}{\partial x_i}  
	\cdot \frac{\partial\boldsymbol{\omega}_\mathrm{N}  }{\partial x_j}   
	\big)\m^\epsilon,
	\end{gathered}
	\end{equation*}
	and one can check the remaining terms satisfy
	\begin{equation*}
	\begin{aligned}
	\Vert \mathbf{D}_2(\boldsymbol{\omega}_\mathrm{N})  
	-
	\widetilde{\mathbf{D}}_2(\boldsymbol{\omega}_\mathrm{N})  \Vert_{L^2(\Omega)} 
	\le
	C ( 1 + \Vert \nabla \m_0 \Vert_{L^{\infty}(\Omega)}^2 )
	\Vert \boldsymbol{\omega}_\mathrm{N} \Vert_{L^{2}(\Omega)}
	\le & 
	C \epsilon\ln (\epsilon^{-1} + 1) .
	\end{aligned}
	\end{equation*}
	Substituting $\m^\epsilon = (\m^\epsilon - \widetilde{\m}^\epsilon - \boldsymbol{\omega}_\mathrm{N}) + (\widetilde{\m}^\epsilon + \boldsymbol{\omega}_\mathrm{N})$ into $\widetilde{\mathbf{D}}_2(\boldsymbol{\omega}_\mathrm{N})$, one can write
	\begin{equation*}
	\begin{aligned}
	\widetilde{\mathbf{D}}_2(\boldsymbol{\omega}_\mathrm{N}) 
	=&
	\mathcal{T}_{\mathrm{bad}}^2 
	+
	\alpha \sum_{i, j= 1}^n
	\big(a^\epsilon_{ij}
	\frac{\partial\boldsymbol{\omega}_\mathrm{N}}{\partial x_i}  
	\cdot \frac{\partial\{\m^\epsilon - \widetilde{\m}^\epsilon - \boldsymbol{\omega}_\mathrm{N}\}}{\partial x_j}  
	\big)\m^\epsilon\\
	& +
	\alpha \sum_{i, j= 1}^n
	\big(a^\epsilon_{ij}
	\frac{\partial\boldsymbol{\omega}_\mathrm{N}}{\partial x_i}  
	\cdot \frac{\partial\{\widetilde{\m}^\epsilon + \boldsymbol{\omega}_\mathrm{N}\}}{\partial x_j}  
	\big)\big(
	\m^\epsilon - \widetilde{\m}^\epsilon 
	- \boldsymbol{\omega}_\mathrm{N}\big).
	\end{aligned}
	\end{equation*}
	Hence it follows that
	\begin{equation*}
	\begin{aligned}
	\Vert
	\widetilde{\mathbf{D}}_2(\boldsymbol{\omega}_\mathrm{N}) 
	- 
	\mathcal{T}_{\mathrm{bad}}^2
	\Vert_{L^2(\Omega)} 
	\le &
	C ( 1 + \Vert \nabla \m_0 \Vert_{L^{\infty}(\Omega)}^2 )
	\Vert \m^\epsilon - \widetilde{\m}^\epsilon - \boldsymbol{\omega}_\mathrm{N} \Vert_{H^{1}(\Omega)}.
	\end{aligned}
	\end{equation*}
	The assertion is proved.
	
	As for \eqref{estimate of partial_t w_b}, one can deduce from Lemma \ref{lemma: estimate of chi} and the proof of Lemma \ref{lemma: estimate of theta and gamma} to obtain
	\begin{equation*}
		\begin{gathered}
		\Vert \partial_t\boldsymbol{\omega}_\mathrm{N}  \Vert_{L^\infty(\Omega)} 
		\le C\epsilon \ln (\epsilon^{-1} + 1),\\
		\Vert \partial_t\boldsymbol{\omega}_\mathrm{M}^i  \Vert_{L^\infty(\Omega)} 
		\le C\Vert \partial_t G^i_l \Vert_{L^\infty(\Omega)}
		\le 
		C\epsilon \ln (\epsilon^{-1} + 1),
		\end{gathered}
	\end{equation*}
where $C$ depends on $\Vert \nabla (\partial_t \m_0) \Vert_{W^{1,\infty}(\Omega)}$.
In order to prove \eqref{estimate of A e_b}, we use Lemma \ref{lemma: estimate of chi}, Lemma \ref{lemma: estimate of theta and gamma} and definition of $\widetilde{\m}^\epsilon$, to deduce the estimates
\begin{equation*}
	\Vert \mathcal{A}_\epsilon \widetilde{\m}^\epsilon \Vert_{L^{2} (\Omega)}
	+
	\Vert \mathcal{A}_\epsilon \boldsymbol{\omega}_{{\mathrm{b}}} \Vert_{L^{2} (\Omega)}
	\le C\ln (\epsilon^{-1} + 1),
\end{equation*}
then \eqref{estimate of A e_b} follows with some constant $C$ depending on $\Vert \mathcal{A}_\epsilon \m^\epsilon \Vert_{L^{2} (\Omega)}$.
Therefore Theorem is proved.
\end{proof}

\section{Stability Analysis}\label{sec:stability}

In this section, we will discuss the stability of following initial-boundary problem, which is motivated by equation \eqref{system of e}:
\begin{equation}\label{eqn:system of e}
	\left\{ \begin{aligned}
		\partial_t\e - \mathcal{L}^\epsilon(\e) 
		&= \boldsymbol{F} \quad \mbox{in $\Omega$},\\
		\boldsymbol{\nu}\cdot a^\epsilon\nabla \e &= \mathbf{g} \quad \mbox{on $\partial\Omega$},\\
		\e(0,x) &= \mathbf{h} \quad \mbox{in $\Omega$} .
	\end{aligned} \right.
\end{equation}
The following two inequalities will be used. The first inequality is motivated by $W^{1,p}$ estimate for oscillatory elliptic problem.
\begin{lemma}\label{lemma: regularity for epslon}
	Assume $u \in H^{2}(\Omega)$, $\boldsymbol{\nu} \cdot a^\epsilon\nabla u = g$ on $\partial \Omega$, with $g\in B^{-1/2,2}(\partial\Omega)$,
	then it holds that for $n\le 3$,
	\begin{equation*}
		\Vert \nabla u \Vert_{L^{6}(\Omega)}
		\le C
		\Vert \mathcal{A}_\epsilon u \Vert_{L^{2}(\Omega)}
		+ C\Vert g \Vert_{B^{-1/2,2}(\partial\Omega)},
	\end{equation*}
	moreover, if $g = 0$, then one has for $n\le 3$
	\begin{equation*}
		\Vert \nabla u \Vert_{L^{6}(\Omega)}
		\le C
		\Vert \mathcal{A}_\epsilon u \Vert_{L^{2}(\Omega)}.
	\end{equation*}
Constant $C$ is independent of $\epsilon$.
\end{lemma}
\begin{proof}
	We refer that Lemma \ref{lemma: regularity for epslon} is a direct corollary of Theorem 6.3.2 in \cite{shen2018periodic}. One can find the proof in \cite{shen2018periodic} [Pages 144-152].
\end{proof}
We also introduce Sobolev inequality with small coefficient when $n=2$.
\begin{lemma}\label{lemma: disturbed Sobolev inequality}
	For any function $f\in H^2(\Omega)$, one has when $n=2$
	\begin{equation*}
		\Vert f \Vert_{L^\infty(\Omega)}
		\le
		C \ln(\epsilon^{-1}+1) \Vert f \Vert_{H^1(\Omega)}
		+ \epsilon \Vert \mathcal{A}^\epsilon f \Vert_{L^2(\Omega)},
	\end{equation*}
where constant $C$ is independent of $\epsilon$.
\end{lemma}
\begin{proof}
	Using the Neumann function, see \cite{shen2018periodic} Section 7.4, one has
	\begin{equation*}
		f 
		= - \int_{\Omega} \nabla_{\bz} N^\epsilon(\bx,\bz)\cdot a^\epsilon \nabla f(\bz) \d\bz
		+ \frac{1}{\abs{\partial\Omega}}\int_{\partial\Omega}f \d \bz
		=: P_1 + P_2.
	\end{equation*}
	Applying cut-off function $\phi^\epsilon = \phi^\epsilon(\bx-\bz)$, the first term yields by integration by parts
	\begin{equation*}
		\begin{aligned}
			P_1 = &
			- \int_{\Omega} (1-\phi^\epsilon) \nabla_{\bz} N^\epsilon(\bx,\bz)\cdot a^\epsilon \nabla f(\bz) \d\bz
			+
			\int_{\Omega} \phi^\epsilon N^\epsilon(\bx,\bz)\cdot \mathcal{A}_\epsilon f(\bz) \d\bz\\
			& +
			\int_{\Omega} \nabla_{\bz}\phi^\epsilon(\bx-\bz)\cdot
			N^\epsilon(\bx,\bz)\cdot a^\epsilon \nabla f(\bz) \d\bz\\
			\le &
			C \epsilon \Vert \mathcal{A}^\epsilon f \Vert_{L^2(\Omega)}
			+
			C \ln(\epsilon^{-1}+1) \Vert \nabla f \Vert_{L^2(\Omega)},
		\end{aligned}
	\end{equation*}
here in the last line we have used the fact $\nabla_{\bz} N^\epsilon(\bx,\bz)\le C\abs{\bx-\by}^{-1}$ and $N^\epsilon(\bx,\bz)\le C\{1+\ln[\abs{\bx-\bz}^{-1}]\}$ for $n=2$, see \cite{shen2018periodic} page 159. As for $P_2$, one has by trace inequality that $P_2 \le C \Vert f \Vert_{H^1(\Omega)}$. The Lemma is proved.
\end{proof}

Now let us give the stability of system \eqref{eqn:system of e} in terms of $\mathbf{h}$, $\mathbf{g}$, $\boldsymbol{F}$ in $L^{\infty}(0,T;L^2(\Omega))$ and $L^{\infty}(0,T;H^1(\Omega))$ norm, respectively.

\subsection{Stability in $L^{\infty}(0,T;L^2(\Omega))$}
\begin{theorem}\label{theorem:stability}
	Let $\e \in L^\infty(0,T;H^2(\Omega))$ be a strong solution to \eqref{eqn:system of e}.
	Assume $\mathbf{h} \in L^{2}(\Omega)$, $\mathbf{g}\in L^{\infty}(0,T; B^{-1/2, 2} (\partial\Omega))$, and $\boldsymbol{F} = \boldsymbol{F}_1 + \boldsymbol{F}_2 $ satisfies
	\begin{equation}\label{assumption of F1 and F2}
		\boldsymbol{F}_1\in L^2(0,T; L^{\sigma} (\Omega)),
		\quad
		\boldsymbol{F}_2
		\in L^2(0,T; L^{2} (\Omega))
	\end{equation}
	with $\sigma=1$ when $n =1,2$, and $\sigma=6/5$ when $n =3$,
	then it holds that, for any $0\le t\le T$
	\begin{equation}\label{eqn:stability 1}
		\begin{aligned}
			&\Vert \e \Vert^2_{L^\infty(0,T;L^2(\Omega))} 
			+ 
			\Vert \nabla \e \Vert^2_{L^2(0,T; L^{2} (\Omega))}\\
			\le  &
			C_{\delta} \Big( 
			\Vert \mathbf{h} \Vert^2_{L^{2}(\Omega)} 
			+
			\Vert \mathbf{g} \Vert^2_{L^2(0,T; B^{-1/2, 2} (\partial\Omega))}
			+
			\gamma(\epsilon)
			\Vert \boldsymbol{F}_1 \Vert^2_{L^2(0,T; L^{\sigma} (\Omega))} \Big) \\
			& + 
			\delta\Vert \boldsymbol{F}_2 \Vert^2_{L^2(0,T; L^{2} (\Omega))}
			+
			\epsilon^2 \Vert \mathcal{A}_\epsilon \e \Vert_{L^2(0,T; L^{2} (\Omega))}^2,
		\end{aligned}
	\end{equation}
	for any small $\delta > 0$, where
	\begin{equation*}
		\left\{\begin{aligned}
					&\gamma(\epsilon)=1,&
					&\text{when $n =1,3$},\\
					&\gamma(\epsilon)=[\ln (\epsilon^{-1} + 1)]^2, & &\text{when $n =2$}.
				\end{aligned}
				\right.
	\end{equation*}
	$C_\delta$ is a constant depending on $\Vert \nabla \m^\epsilon  \Vert_{L^{4}(\Omega)}$, $\Vert \nabla \widetilde{\m}^\epsilon  \Vert_{L^{4}(\Omega)}$, but is independent of $t$ and $\epsilon$.
\end{theorem}

\begin{proof}
	The inner product between \eqref{eqn:system of e} and $\e$ in $L^2(\Omega)$ leads to
	\begin{equation}
		\begin{aligned}\label{eqn:product with e}
		&\frac{1}{2} \frac{\d}{\d t} \int_{\Omega} \vert \e \vert^2 \d \bx
		-
		\alpha  \int_{\Omega} \widetilde{\mathcal{H}}^\epsilon_e(\e) \cdot \e \d \bx\\
		= & 
		- \int_{\Omega}
		\mathbf{D}_1(\e)  \cdot \e \d \bx 
		-
		\int_{\Omega}
		\mathbf{D}_2(\e) \cdot \e\d \bx
		-
		\int_{\Omega} \boldsymbol{F} \cdot \e \d \bx.
	\end{aligned}
\end{equation}
	Now let us give the estimates to \eqref{eqn:product with e} term by term. Integration by parts for the second term on the left-hand side yields
	\begin{equation*}
		- \int_{\Omega} \widetilde{\mathcal{H}}^\epsilon_e(\e) \cdot \e \d \bx
		\ge
		\sum_{i, j= 1}^{n}
		\int_{\Omega} a^\epsilon_{ij} \frac{\partial\e}{\partial x_i} \cdot
		\frac{\partial\e}{\partial x_j} \d \bx
		-
		\int_{\partial\Omega} \mathbf{g} \cdot  \e \d \bx 
		- 
		C \int_{\Omega} \vert \e \vert^2 \d \bx
	\end{equation*}
	with the boundary term satisfying
	\begin{equation*}
		\begin{aligned}
			\int_{\partial\Omega} \mathbf{g} \cdot  \e \d \bx
			\le
			\Vert \mathbf{g}\Vert_{B^{-1/2,2}(\partial\Omega)}
			\Vert \e\Vert_{B^{1/2,2}(\partial\Omega)}
			\le 
			C\Vert \mathbf{g}\Vert_{B^{-1/2,2}(\partial\Omega)}
			\Vert \e\Vert_{H^1(\Omega)}.
		\end{aligned}
	\end{equation*}
	By integration by parts and the same argument for the boundary term, the first term on the right-hand side can be estimated as
	\begin{equation*}\label{bounded of D_1 e cdot e}
		\begin{aligned}
			- \int_{\Omega}
			\mathbf{D}_1(\e)  \cdot \e \d \bx 
			\le
			& C \int_{\Omega} \vert \e \vert^2 \d \bx
			+
			\delta C \int_{\Omega} \vert \nabla \e \vert^2 \d \bx
			-
			C\Vert \mathbf{g}\Vert_{B^{-1/2,2}(\partial\Omega)}
			\Vert \e\Vert_{H^1(\Omega)}.
		\end{aligned}
	\end{equation*}
	For the second term on the right-hand side of \eqref{eqn:product with e}, using the estimates
	\begin{equation*}
		\begin{aligned}
			&\int_{\Omega}
			\big( \mathcal{B}^\epsilon[\e,\m^\epsilon ] 
			\big)\m^\epsilon \cdot \e \d \bx
			\le  C
			\Vert \nabla \m^\epsilon \Vert_{L^{4}(\Omega)}
			\Vert \e \Vert_{L^{4}(\Omega)}
			\Vert \nabla \e \Vert_{L^{2}(\Omega)} \\
			&\qquad\qquad\qquad\qquad\qquad\qquad+
			C
			\Vert \m^\epsilon \Vert_{L^{4}(\Omega)}
			\Vert \e \Vert_{L^{4}(\Omega)}
			\Vert \e \Vert_{L^{2}(\Omega)},\\
			&\int_{\Omega} g_l^\epsilon[\widetilde{\m}^\epsilon]  \e \cdot \e \d  x
			\le  C
			\Vert \nabla \widetilde{\m}^\epsilon \Vert_{L^{4}(\Omega)}^2
			\Vert \e \Vert_{L^{4}(\Omega)}^2
			+ C
			\Vert \widetilde{\m}^\epsilon \Vert_{L^{4}(\Omega)}^2
			\Vert \e \Vert_{L^{4}(\Omega)}^2,
		\end{aligned}
	\end{equation*}
	and the same argument can be applied to the other terms, we finally obtain by Sobolev inequality
	\begin{equation*}\label{bounded of D_2 e cdot e}
		\begin{aligned}
			- \int_{\Omega}
			\mathbf{D}_2(\e)  \cdot \e \d \bx 
			\le C +
			C \int_{\Omega} \vert \e \vert^2 \d \bx
			+
			\delta C \int_{\Omega} \vert \nabla \e \vert^2 \d \bx,
		\end{aligned}
	\end{equation*}
	where $C = C^0\big(1 + \Vert \nabla \m^\epsilon  \Vert_{L^{4}(\Omega)}^2 + \Vert \nabla \widetilde{\m}^\epsilon  \Vert_{L^{4}(\Omega)}^2\big) $.
	For the last term in \eqref{eqn:product with e}, by the assumption \eqref{assumption of F1 and F2}, we apply Sobolev inequality for $n=1,3$, and apply Lemma \ref{lemma: disturbed Sobolev inequality} for $n=2$, it follows that
	\begin{equation*}
		\begin{aligned}
			-
			\int_{\Omega} \boldsymbol{F}_1 \cdot \e \d \bx
			\le &
			\left\{\begin{aligned}
			& C
			\Vert \boldsymbol{F}_1 \Vert_{L^{1}(\Omega)}^2
			+ \delta\Vert \e \Vert_{H^1(\Omega)}^2,& n=1& \\
			& \begin{aligned}&C
			[\ln(\epsilon^{-1}+1)]^2
			\Vert \boldsymbol{F}_1 \Vert_{L^{1}(\Omega)}^2\\
			&\qquad\qquad+ \delta\Vert \e \Vert_{H^1(\Omega)}^2
			+ \epsilon^2 \Vert \mathcal{A}^\epsilon \e \Vert_{L^2}^2,
			\end{aligned} & n=2& \\
			& C
			\Vert \boldsymbol{F}_1 \Vert_{L^{6/5}(\Omega)}^2
			+ \delta\Vert \e \Vert_{H^1(\Omega)}^2,& n=3&
			\end{aligned}\right.\\
			-
			\int_{\Omega} \boldsymbol{F}_2 \cdot \e \d \bx
			\le& \delta^*
			\Vert \boldsymbol{F}_2 \Vert_{L^2(\Omega)}^2
			+ C
			\Vert \e \Vert_{L^2(\Omega)},
		\end{aligned}
	\end{equation*}
	with any small $\delta$, $\delta^*>0$. Substituting above estimates, one can derive from \eqref{eqn:product with e} that
	\begin{equation*}\label{increasing inequality}
		\begin{aligned}
			\frac{1}{2}\frac{\d}{\d t} \int_{\Omega} \vert \e \vert^2 \d \bx
			+
			(\alpha a_{\mathrm{min}} - 2 \delta) \int_{\Omega} \abs{ \nabla \e }^2 \d \bx
			\le  
			C \int_{\Omega} \vert \e \vert^2 \d \bx
			+ 
			C \Vert \mathbf{g} \Vert^2_{B^{-1/2, 2} (\partial\Omega)}\\
			+ 
			C [\ln(\epsilon^{-1}+1)]^2
			\Vert \boldsymbol{F}_1 \Vert^2_{L^{\sigma} (\Omega)} 
			+
			\delta^*\Vert \boldsymbol{F}_2 \Vert^2_{L^{2} (\Omega)} 
			+
			\epsilon^2 \Vert \mathcal{A}^\epsilon \e \Vert_{L^2}^2.
		\end{aligned}
	\end{equation*}
	Then \eqref{eqn:stability 1} follows directly by taking $\delta$ small enough, and the application of Gr\"{o}nwall's inequality.
\end{proof}

\subsection{Stability in $L^{\infty}(0,T;H^1(\Omega))$}
\begin{theorem}\label{theorem:stability 2}
	Let $\e \in L^\infty(0,T;H^2(\Omega))$ be a strong solution to \eqref{eqn:system of e}. Assume $\mathbf{h} \in H^{1}(\Omega)$, $\mathbf{g}\in H^{1}(0,T; B^{-1/2, 2} (\partial\Omega))$, and $\boldsymbol{F}\in L^2(0,T; L^{2} (\Omega))$,
	it holds
	\begin{equation}\label{eqn:stability}
		\begin{aligned}
			\Vert \nabla \e \Vert^2_{L^\infty(0,T;L^2(\Omega))} 
			\le  &
			C \big( \Vert \mathbf{h} \Vert^2_{H^{1}(\Omega)} 
			+
			\Vert \mathbf{g} \Vert^2_{H^{1}(0,T; B^{-1/2, 2} (\partial\Omega))}
			+
			\Vert \boldsymbol{F} \Vert^2_{L^2(0,T;L^2(\Omega))}  \big) ,
		\end{aligned}
	\end{equation}
	where
	$C$ depends on $\Vert \nabla \m^\epsilon  \Vert_{L^{4}(\Omega)}$, $\Vert \nabla \widetilde{\m}^\epsilon  \Vert_{L^{4}(\Omega)}$ and $\Vert \mathcal{H}^\epsilon_e (\widetilde{\m}^\epsilon) \Vert_{L^{4}(\Omega)}$, but is independent of $t$ and $\epsilon$.
\end{theorem}

\begin{proof}
	The inner product between \eqref{eqn:system of e} and $\widetilde{\mathcal{H}}^\epsilon_e(\e)$ in $L^2(\Omega)$ leads to
	\begin{equation}\label{eqn:product with H}
		\begin{aligned}
			& - \int_{\Omega} \partial_t \e \cdot \widetilde{\mathcal{H}}^\epsilon_e(\e) \d \bx
			+
			\alpha  \int_{\Omega} \widetilde{\mathcal{H}}^\epsilon_e(\e) \cdot \widetilde{\mathcal{H}}^\epsilon_e(\e) \d \bx\\
			= &
			\int_{\Omega}
			\mathbf{D}_1 (\e)  \cdot \widetilde{\mathcal{H}}^\epsilon_e(\e) \d \bx 
			+
			\int_{\Omega}
			\mathbf{D}_2(\e) \cdot \widetilde{\mathcal{H}}^\epsilon_e(\e) \d \bx
			+
			\int_{\Omega}
			\boldsymbol{F} \cdot \widetilde{\mathcal{H}}^\epsilon_e(\e) \d \bx.
		\end{aligned}
	\end{equation}
	In the following we give the estimates to \eqref{eqn:product with H} term by term. 
	Note that integration by parts yields
	\begin{align*}
		- \int_{\Omega} \partial_t \e \cdot \widetilde{\mathcal{H}}^\epsilon_e(\e) \d \bx
		= &
		\frac{\d}{\d t} \G^\epsilon[\e]
		-
		\int_{\partial\Omega} \partial_t \e \cdot \mathbf{g} \d \bx, \\
		= &
		\frac{\d}{\d t} \G^\epsilon[\e]
		- \partial_t \int_{\partial\Omega} \e \cdot \mathbf{g} \d \bx 
		-
		C\Vert \partial_t \mathbf{g}\Vert_{B^{-1/2,2}(\partial\Omega)}
		\Vert \e\Vert_{H^1(\Omega)}.
	\end{align*}
	Using the fact $\m^\epsilon\times \widetilde{\mathcal{H}}^\epsilon_e(\e) \cdot\widetilde{\mathcal{H}}^\epsilon_e(\e)=0$, the first term on the right-hand side of \eqref{eqn:product with H} can be estimate by Sobolev inequality as 
	\begin{align*}
		\int_{\Omega}
		\mathbf{D}_1 (\e) \cdot \widetilde{\mathcal{H}}^\epsilon_e(\e) \d \bx 
		&\le 
		C \Vert \mathcal{H}^\epsilon_e (\widetilde{\m}^\epsilon) \Vert_{L^{4}(\Omega)}
		\Vert \e \Vert_{L^{4}(\Omega)}
		\Vert \widetilde{\mathcal{H}}^\epsilon_e(\e) \Vert_{L^{2}(\Omega)} \nn\\
		&\le
		C \Vert \e \Vert_{H^{1}(\Omega)}^2
		+
		\delta C\Vert \widetilde{\mathcal{H}}^\epsilon_e(\e) \Vert_{L^{2}(\Omega)}^2, \label{bounded of D_1 e dot H e}
	\end{align*}
	where $C = C^0\big(1 + \Vert \mathcal{H}^\epsilon_e (\widetilde{\m}^\epsilon) \Vert_{L^{4}(\Omega)}^2 \big) $. For the second term on the right-hand side of \eqref{eqn:product with H}, note that we have the estimate
	\begin{equation*}
		\begin{aligned}
			\int_{\Omega}
			\big( \mathcal{B}^\epsilon[\e,\m^\epsilon ] 
			\big)\m^\epsilon 
			\cdot \widetilde{\mathcal{H}}^\epsilon_e(\e) \d \bx
			\le & C
			\Vert \nabla \m^\epsilon  \Vert_{L^{4}(\Omega)}
			\Vert \nabla \e \Vert_{L^{4}(\Omega)}
			\Vert \widetilde{\mathcal{H}}^\epsilon_e(\e) \Vert_{L^{2}(\Omega)} \\
			& + C
			\Vert \m^\epsilon \Vert_{L^{4}(\Omega)}
			\Vert \e \Vert_{L^{4}(\Omega)}
			\Vert \widetilde{\mathcal{H}}^\epsilon_e(\e) \Vert_{L^{2}(\Omega)},
		\end{aligned}
	\end{equation*}
	in which 
	one can deduce
	\begin{equation*}
		\begin{aligned}
			\Vert \nabla \e \Vert_{L^{4}(\Omega)}
			\le &
			\Vert \nabla \e \Vert_{L^{2}(\Omega)}^{1/3} 
			\Vert \nabla \e \Vert_{L^{6}(\Omega)}^{2/3} \\
			\le & C
			\Vert \nabla \e \Vert_{L^{2}(\Omega)}^{1/3} 
			(1 + \Vert \widetilde{\mathcal{H}}^\epsilon_e(\e) \Vert_{L^{2}(\Omega)} 
			+ 
			\Vert \mathbf{g}\Vert_{B^{-1/2,2}(\partial\Omega)})^{2/3},
		\end{aligned}
	\end{equation*}
	using interpolation inequality and Lemma \ref{lemma: regularity for epslon}. The other terms can be estimated in the same fashion. 
	After the application of Young's inequality, one finally obtains
	\begin{equation*}\label{bounded of D_2 e dot H e}
		\int_{\Omega}
		\mathbf{D}_2 (\e) \cdot \widetilde{\mathcal{H}}^\epsilon_e(\e) \d \bx 
		\le 
		C \Vert \e \Vert_{H^{1}(\Omega)}^2
		+
		\delta C\Vert \widetilde{\mathcal{H}}^\epsilon_e(\e) \Vert_{L^{2}(\Omega)}^2
		+
		C \Vert \mathbf{g}\Vert_{B^{-1/2,2}(\partial\Omega)}^2,
	\end{equation*}
	where $C = C^0(1 + \Vert \nabla \m^\epsilon  \Vert_{L^{4}(\Omega)}^2 + \Vert \nabla \widetilde{\m}^\epsilon  \Vert_{L^{4}(\Omega)}^2)$.
	Substituting above estimates into \eqref{eqn:product with H}, 
	we arrive at
	\begin{align*}
		\frac{\d}{\d t} \G^\epsilon[\e]
		+ &
		(\alpha - C \delta)\int_{\Omega}
		\vert \widetilde{\mathcal{H}}^\epsilon_e(\e) \vert^2 \d \bx
		- 
		\partial_t \int_{\partial\Omega} \e \cdot \mathbf{g} \d \bx\\
		\le &
		C \big(\Vert \e \Vert_{H^1(\Omega)}^2 
		+
		\Vert \boldsymbol{F} \Vert_{L^{2}(\Omega)}^2 
		+
		\Vert \mathbf{g}\Vert_{B^{-1/2,2}(\partial\Omega)}
		+
		\Vert \partial_t \mathbf{g}\Vert_{B^{-1/2,2}(\partial\Omega)}\big).
	\end{align*}
	Integrating the above inequality over $[0,t]$ with $0<t< T$ and using the facts
	\begin{align*}
		\int_{\partial\Omega} \e \cdot \mathbf{g} \d \bx
		\le &
		\delta \Vert \e\Vert_{H^{1}(\Omega)}^2
		+ C
		\Vert \mathbf{g}\Vert_{B^{-1/2,2}(\partial\Omega)}^2, \\
		\G^\epsilon[\e]
		\ge &
		\frac{a_{\mathrm{min}}}{2}\Vert \nabla \e \Vert_{L^{2}(\Omega)}^2 - C,
	\end{align*}
	one can finally derive
	\begin{multline*}
		(\frac{a_{\mathrm{min}}}{2} - \delta )\Vert \nabla \e(t) \Vert_{L^{2}(\Omega)}^2
		+ 
		(\alpha - C \delta)\int_0^t
		\Vert \widetilde{\mathcal{H}}^\epsilon_e(\e) \Vert_{L^2(\Omega)}^2 \d \tau \\
		\le 
		C \int_0^t \big(\Vert \e \Vert_{L^2(\Omega)}^2 
		+
		\Vert \nabla\e \Vert_{L^2(\Omega)}^2 
		+
		\Vert \boldsymbol{F} \Vert_{L^{2}(\Omega)}^2 
		+
		\Vert \partial_t \mathbf{g}\Vert_{B^{-1/2,2}(\partial\Omega)}
		 \big) \d \tau
		+ \mathcal{J}(\mathbf{h}),
	\end{multline*}
	where $\mathcal{J}(\mathbf{h})$ yields
	\begin{equation*}
		\mathcal{J}(\mathbf{h}) 
		=
		\G^\epsilon[\mathbf{h}]
		-
		\int_{\partial\Omega} \mathbf{h}
		\cdot  \frac{\partial}{\partial \boldsymbol{\nu}^\epsilon}
		\mathbf{h}  \d \bx 
		\le 
		C \Vert \mathbf{h} \Vert_{H^{1}(\Omega)}^2
		+
		\Vert \frac{\partial}{\partial \boldsymbol{\nu}^\epsilon}
		\mathbf{h}\Vert_{B^{-1/2,2}(\partial\Omega)}^2. 
	\end{equation*}
	\eqref{eqn:stability} is then derived after taking $\delta$ small enough and the application of Gr\"{o}nwall's inequality.
\end{proof}

\section{Regularity}\label{Preliminary} 
In the estimate of boundary corrector and stability analysis by Theorem \ref{theorem:omega}, Theorem \ref{theorem:stability}, Theorem \ref{theorem:stability 2}, the constant we deduced rely on the value of $\Vert \mathcal{A}_\epsilon \m^\epsilon \Vert_{L^{2}(\Omega)}$ and $\Vert \nabla \m^\epsilon\Vert_{L^{6}(\Omega)}$.
In this section, we introduce the uniform regularity on $\m^\epsilon$, over a time interval independent of $\epsilon$. For this purpose, we intend to derive a structure-preserving energy inequality, in which the degenerate term are kept in the energy. 

First, let us introduce an interpolation inequality of the effective field $\mathcal{H}^\epsilon_e(\m^\epsilon)$ for some $S^2$-valued function $\m^\epsilon$, which is the generalization of \eqref{term of high order in intro}.
The following estimates will be used:
\begin{gather}\label{eqn:nabla m dot nabla m}
	a_{\mathrm{max}}^{-1}
	\Vert \m^\epsilon \cdot\mathcal{A}_\epsilon \m^\epsilon \Vert_{L^{3}(\Omega)}^3
	\le
	\Vert \nabla \m^\epsilon \Vert_{L^{6}(\Omega)}^6
	\le 
	a_{\mathrm{min}}^{-1}
	\Vert \mathcal{A}_\epsilon \m^\epsilon \Vert_{L^{3}(\Omega)}^3,\\
	\label{relation between A and H}
	\Vert \mathcal{A}_\epsilon \m^\epsilon \Vert_{L^{p}(\Omega)} - C_p
	\le 
	\Vert \mathcal{H}^\epsilon_e(\m^\epsilon) \Vert_{L^{p}(\Omega)}
	\le 
	\Vert \mathcal{A}_\epsilon \m^\epsilon \Vert_{L^{p}(\Omega)} + C_p,
\end{gather}
with $1<p<+\infty$, here the first line follows from the fact $- a^\epsilon \abs{\nabla \m^\epsilon}^2 = \m^\epsilon \cdot\mathcal{A}_\epsilon \m^\epsilon$ by $\abs{\m^\epsilon} = 1$ and assumption of $a^\epsilon$ in \eqref{uniformly coercive}, and in second line the estimate \eqref{bound of stray field} is used. We also introduce a orthogonal decomposition to any vector $\mathbf{a}$ as
\begin{equation}\label{eqn: orthogonality decomposition}
	\mathbf{a}
	=
	(\m^\epsilon \cdot \mathbf{a}) \m^\epsilon
	-
	\m^\epsilon \times (\m^\epsilon \times \mathbf{a}).
\end{equation}

\begin{lemma}\label{lemma:estimate of H in L3}
	Given $\m^\epsilon\in H^3(\Omega)$  that satisfies $\abs{\m^\epsilon} = 1$ and Neumann boundary condition $\boldsymbol{\nu}\cdot a^\epsilon\nabla \m^\epsilon = 0$, then it holds for $n\le 3$ and any $0<\delta<1$
	\begin{equation}\label{term of high order}
		\Vert \mathcal{H}^\epsilon_e(\m^\epsilon) \Vert_{L^{3}(\Omega)}^{3}
		\le 
		C_\delta + C_\delta\Vert \mathcal{H}^\epsilon_e (\m^\epsilon) \Vert_{L^{2}(\Omega)}^{6}
		+ \delta
		\Vert \m^\epsilon \times \nabla \mathcal{H}^\epsilon_e(\m^\epsilon) \Vert_{L^{2}(\Omega)}^{2},
	\end{equation}
	where $C_\delta$ is a constant depending on $\delta$ but independent of $\epsilon$.
\end{lemma}

\begin{proof}
	Applying decomposition \eqref{eqn: orthogonality decomposition} by taking $\mathbf{a} = \mathcal{H}^\epsilon_e(\m^\epsilon)$, one can write
	\begin{equation}\label{case of n le 2}
		\begin{aligned}
			\Vert \mathcal{H}^\epsilon_e(\m^\epsilon) \Vert_{L^{3}(\Omega)}^{3}
			\le &
			\int_{\Omega}   
			\vert \m^\epsilon \cdot\mathcal{H}^\epsilon_e(\m^\epsilon) \vert^3  \d \bx\\
			& +
			\int_{\Omega}
			\vert \m^\epsilon \times \mathcal{H}^\epsilon_e(\m^\epsilon) \vert^3  \d \bx
			=: \mathcal{I}_1 + \mathcal{I}_2.
		\end{aligned}
	\end{equation}
	Now let us estimate the right-hand side of \eqref{case of n le 2} separately. For $\mathcal{I}_1$, we apply \eqref{eqn:nabla m dot nabla m} and Remark \ref{lemma: regularity for epslon} to derive
	\begin{equation*}
		\begin{aligned}
			\mathcal{I}_1
			\le &
			C + C
			\Vert \nabla \m^\epsilon \Vert_{L^{6}(\Omega)}^6  
			\le 
			C + C\Vert \mathcal{H}^\epsilon_e (\m^\epsilon) \Vert_{L^{2}(\Omega)}^{6}.
		\end{aligned}
	\end{equation*}
	As for $\mathcal{I}_2$, we have by Sobolev inequality for $n\le 3$
	\begin{align*}
		& \mathcal{I}_2
		\leq C +C\Vert \m^\epsilon \times \mathcal{H}^\epsilon_e(\m^\epsilon) \Vert_{L^{2}(\Omega)}^{6}
		+\delta^*\Vert \m^\epsilon \times \mathcal{H}^\epsilon_e(\m^\epsilon) \Vert_{H^{1}(\Omega)}^{2}, \label{eqn:process in regurality 2}
	\end{align*}
	here in the last term, we can apply \eqref{eqn:nabla m dot nabla m}-\eqref{relation between A and H} to derive:
	\begin{equation*}\label{eqn:process in regurality}
		\begin{aligned}
			\delta^* \Vert \nabla \m^\epsilon \times \mathcal{H}^\epsilon_e(\m^\epsilon) \Vert_{L^{2}(\Omega)}^{2}
			\le &
			\delta^* \Vert \nabla \m^\epsilon \Vert_{L^{6}(\Omega)}^{6}
			+
			\delta^* \Vert \mathcal{H}^\epsilon_e(\m^\epsilon) \Vert_{L^{3}(\Omega)}^{3}\\
			\le &
			C + C\delta^* \Vert \mathcal{H}^\epsilon_e(\m^\epsilon) \Vert_{L^{3}(\Omega)}^{3}.
		\end{aligned}
	\end{equation*}
	Now let us turn back to \eqref{case of n le 2}, we finally obtain
	\begin{equation*}
		\begin{aligned}
			(1 - C \delta^*)\Vert \mathcal{H}^\epsilon_e(\m^\epsilon) \Vert_{L^{3}(\Omega)}^{3}
			\le &
			C + C\Vert \mathcal{H}^\epsilon_e (\m^\epsilon) \Vert_{L^{2}(\Omega)}^{6}
			+ \delta^*
			\Vert \m^\epsilon \times \nabla \mathcal{H}^\epsilon_e(\m^\epsilon) \Vert_{L^{2}(\Omega)}^{2}.
		\end{aligned}
	\end{equation*}
	Let $\delta^* < \frac{1}{2C}$, one can derive \eqref{term of high order} with $\delta = \delta^*/(1 - C \delta^*)< 1$.
\end{proof}

Now let us recall some energy property of LLG equation, and give the uniform regularity result. Using the formula of vector outer production
\begin{equation}\label{eqn:vec product}
	\mathbf{a}\times (\mathbf{b} \times \mathbf{c}) = (\mathbf{a} \cdot \mathbf{c}) \mathbf{b} - (\mathbf{a}\cdot\mathbf{b}) \mathbf{c},
\end{equation}
one can rewrite LLG equation \eqref{eqn:LLG system form 3} into a degenerate form
\begin{equation}\label{eqn:LLG system form 2}
	\partial_t\m^\epsilon 
	+ \alpha \m^\epsilon \times 
	\big(\m^\epsilon \times \mathcal{H}^\epsilon_e(\m^\epsilon) \big) 
	+ \m^\epsilon \times \mathcal{H}^\epsilon_e(\m^\epsilon) 
	=0.
\end{equation}
Multiplying \eqref{eqn:LLG system form 2} by $\mathcal{H}^\epsilon_e(\m^\epsilon)$ and integrating over $(0,t)$, we derive the energy dissipation identity
\begin{equation}\label{energy decay}
	\G^\epsilon[\m^\epsilon(t)] + 
	\alpha \int_0^t \Vert \m^\epsilon \times \mathcal{H}^\epsilon_e(\m^\epsilon) \Vert_{L^2(\Omega)}^2 \d \tau = \G^\epsilon[\m^\epsilon(0)],
\end{equation}
together \eqref{eqn:LLG system form 2} and \eqref{energy decay} leads to the integrable of kinetic energy
\begin{align}\label{integrable of kinetic energy}
	\frac{\alpha}{1+\alpha^2}\int_0^t \Vert \partial_t \m^\epsilon \Vert_{L^2(\Omega)}^2 \d \tau &
	\le 
	\G^\epsilon[\m^\epsilon(0)].
\end{align}
The energy identity \eqref{energy decay} implies the uniform regularity of $\Vert \m^\epsilon \times \mathcal{A}_\epsilon\m^\epsilon \Vert_{L^2(\Omega)}^2$, however, this is not enough to obtain the regularity of $\Vert \mathcal{A}_\epsilon\m^\epsilon \Vert_{L^2(\Omega)}^2$ due to the degeneracy. In this end we introduce that:
%

\begin{theorem}\label{regularety 3}
	Let $\m^\epsilon\in L^2([0,T];H^3(\Omega))$ be a solution to \eqref{eqn:LLG system}. Assume $n\le 3$, then there exists $T^*\in(0, T]$ independent of $\epsilon$, such that for $0\le t\le T^*$,
	\begin{equation*}
		\Vert \mathcal{A}_\epsilon \m^\epsilon  (t) \Vert_{L^{2}(\Omega)}^2 
		+
		\int_0^{t} \big\Vert \m^\epsilon \times \nabla \mathcal{H}^\epsilon_e(\m^\epsilon)  
				(\tau) \big\Vert_{L^{2}(\Omega)}^2 \d \tau  \le C,
	\end{equation*}
	and therefore, by the Sobolev-type inequality in Remark \ref{lemma: regularity for epslon},
	\begin{equation*}
		\Vert \nabla \m^\epsilon  (\cdot, t)\Vert_{L^{6}(\Omega)}^2
		\le
		C,
	\end{equation*}
	where $C$ is a constant independent of $\epsilon$ and $t$.
	
\end{theorem}

\begin{proof}
	Applying $\nabla$ to \eqref{eqn:LLG system form 2} and multiplying by $ a^\epsilon \nabla \mathcal{H}^\epsilon_e(\m^\epsilon)$ lead to
	\begin{equation}\label{eqn:regu esti in 3d}
		\begin{aligned}
			& -
			\int_{\Omega} \nabla (\partial_t \m^\epsilon) \cdot a^\epsilon \nabla \mathcal{H}^\epsilon_e(\m^\epsilon) \d \bx \\
			= & 
			\alpha \int_{\Omega}  \nabla 
			\big(\m^\epsilon\times(\m^\epsilon\times \mathcal{H}^\epsilon_e(\m^\epsilon)) \big)
			\cdot a^\epsilon \nabla \mathcal{H}^\epsilon_e(\m^\epsilon)  \d \bx\\
			& +
			\sum_{i,j=1}^{n}
			\int_{\Omega} \frac{\partial}{\partial x_i} \m^\epsilon \times \mathcal{H}^\epsilon_e(\m^\epsilon) 
			\cdot a^\epsilon_{ij} \frac{\partial}{\partial x_j} \mathcal{H}^\epsilon_e(\m^\epsilon) \d \bx  
			=: 
			\mathcal{J}_1 + \mathcal{J}_2.
		\end{aligned}
	\end{equation}
	Denote $\Gamma^\epsilon(\m^\epsilon) = \mathcal{H}^\epsilon_e(\m^\epsilon)  - \mathcal{A}_\epsilon\m^\epsilon$. After integration by parts, the left-hand side of \eqref{eqn:regu esti in 3d} becomes
	\begin{equation*}
		\begin{aligned}
			& -
			\int_{\Omega} \nabla (\partial_t \m^\epsilon) \cdot a^\epsilon \nabla \mathcal{H}^\epsilon_e(\m^\epsilon) \d \bx 
			= 
			\int_{\Omega} \mathcal{A}_\epsilon (\partial_t \m^\epsilon) \cdot \big( \mathcal{A}_\epsilon\m^\epsilon + \Gamma^\epsilon(\m^\epsilon) \big) \d \bx,
		\end{aligned}
	\end{equation*}
	where the right-hand side can be rewritten as 
	\begin{equation*}\label{eqn:regu esti in 3d 2}
		\begin{aligned}
			\frac{1}{2}\frac{\d }{\d t}
			\int_{\Omega} \vert \mathcal{A}_\epsilon \m^\epsilon \vert^2 \d \bx 
			+
			\frac{\d }{\d t}
			\int_{\Omega} \mathcal{A}_\epsilon  \m^\epsilon \cdot \Gamma^\epsilon(\m^\epsilon) \d \bx  
			-
			\int_{\Omega} \mathcal{A}_\epsilon (\m^\epsilon) \cdot \Gamma^\epsilon(\partial_t \m^\epsilon) \d \bx.
		\end{aligned}
	\end{equation*}
	Now let us consider the right-hand side of \eqref{eqn:regu esti in 3d}. For $\mathcal{J}_1$, one can derive by swapping the order of mixed product
	\begin{equation*}\label{eqn:regu esti in 3d 3}
		\begin{aligned}
			\mathcal{J}_1
			=  -
			\alpha \sum_{i, j= 1}^n\int_{\Omega}   
			\big(\m^\epsilon\times \frac{\partial}{\partial x_i} \mathcal{H}^\epsilon_e(\m^\epsilon)\big) 
			\cdot 	a^\epsilon_{ij} \big(\m^\epsilon\times 
			\frac{\partial}{\partial x_j}\mathcal{H}^\epsilon_e(\m^\epsilon)\big)  \d \bx 
			+ F_1 ,
		\end{aligned}
	\end{equation*}
	here the first term on right-hand side is sign-preserved due to the uniform coerciveness of $a^\epsilon$ in \eqref{uniformly coercive}.
	As for $\mathcal{J}_2$, we apply \eqref{eqn: orthogonality decomposition} by taking $\mathbf{a} = a^\epsilon_{ij} \partial_j \mathcal{H}^\epsilon_e(\m^\epsilon)$, it leads to
	\begin{equation}\label{F_2}
		\begin{aligned}
			\mathcal{J}_2 = & \sum_{i,j=1}^{n}
			\int_{\Omega} \m^\epsilon \times\big( \frac{\partial}{\partial x_i} \m^\epsilon \times \mathcal{H}^\epsilon_e(\m^\epsilon) \big) 
			\cdot
			\big( \m^\epsilon \times  
			a^\epsilon_{ij} \frac{\partial}{\partial x_j} \mathcal{H}^\epsilon_e(\m^\epsilon) \big) \d \bx\\
			& - \sum_{i,j=1}^{n}
			\int_{\Omega} \big( \m^\epsilon \times \mathcal{H}^\epsilon_e(\m^\epsilon) \cdot \frac{\partial}{\partial x_i} \m^\epsilon  \big)
			\m^\epsilon  
			\cdot a^\epsilon_{ij} \frac{\partial}{\partial x_j} \mathcal{H}^\epsilon_e(\m^\epsilon) \d \bx.
		\end{aligned}
	\end{equation}
	Using property of vector outer production \eqref{eqn:vec product} for first term, and integration by parts for the second term, \eqref{F_2} becomes
	\begin{equation*}\label{eqn:transform with F_1}
		\begin{aligned}
			\mathcal{J}_2 =& 2 \sum_{i,j=1}^{n}
			\int_{\Omega} \big( \m^\epsilon \cdot \mathcal{H}^\epsilon_e(\m^\epsilon) \big)
			\big( \m^\epsilon \times  
			a^\epsilon_{ij} \frac{\partial}{\partial x_j} \mathcal{H}^\epsilon_e(\m^\epsilon)
			\cdot \frac{\partial}{\partial x_i} \m^\epsilon \big) \d \bx + F_2\\
			\le & C
			\Vert \nabla \m^\epsilon \Vert_{L^{6}(\Omega)}^6
			+
			C
			\Vert \mathcal{H}^\epsilon_e(\m^\epsilon) \Vert_{L^{3}(\Omega)}^3
			+
			\delta \Vert \m^\epsilon \times \nabla \mathcal{H}^\epsilon_e(\m^\epsilon) \Vert_{L^{2}(\Omega)}^{2} + F_2.
		\end{aligned}
	\end{equation*}
	Here low-order terms $F_i$, $i=1,2$ satisfies by \eqref{relation between A and H} and H\"{o}lder's inequality
	\begin{equation*}\label{bound of F_1 process result}
		\begin{aligned}
			F_i
			\le & C + C
			\Vert \nabla \m^\epsilon \Vert_{L^{6}(\Omega)}^6
			+
			C
			\Vert \mathcal{H}^\epsilon_e(\m^\epsilon) \Vert_{L^{3}(\Omega)}^3.
		\end{aligned}
	\end{equation*}
	Substituting above estimates into \eqref{eqn:regu esti in 3d}, applying estimate \eqref{eqn:nabla m dot nabla m} and Lemma \ref{lemma:estimate of H in L3}, we finally arrive at
	\begin{equation}\label{esimate of regular in 3d}
		\begin{aligned}
			& \frac{1}{2}\frac{\d }{\d t}
			\Vert \mathcal{A}_\epsilon  \m^\epsilon \Vert_{L^2(\Omega)}^2 
			+
			(\alpha a_{\mathrm{min}} - C\delta) \Vert \m^\epsilon \times \nabla \mathcal{H}^\epsilon_e(\m^\epsilon) \Vert_{L^{2}(\Omega)}^{2}  \\
			\le & C +
			C \Vert \mathcal{A}_\epsilon  \m^\epsilon \Vert_{L^2(\Omega)}^6 
			+ C
			\Vert \partial_t \m^\epsilon \Vert_{L^2(\Omega)}^2
			-
			\frac{\d }{\d t}
			\int_{\Omega} \mathcal{A}_\epsilon  \m^\epsilon \cdot \Gamma^\epsilon(\m^\epsilon) \d \bx ,
		\end{aligned}
	\end{equation}
	Integrating \eqref{esimate of regular in 3d} over $[0,t]$, using the integrability of kinetic energy \eqref{integrable of kinetic energy} and the following inequality
	\begin{align*}
		\int_{\Omega} \mathcal{A}_\epsilon  \m^\epsilon \cdot \Gamma^\epsilon(\m^\epsilon) \d \bx & 
		\le
		C\Vert \Gamma^\epsilon(\m^\epsilon)  \Vert_{L^{2}(\Omega)}^{2}
		+ \frac{1}{4}\Vert \mathcal{A}_\epsilon  \m^\epsilon \Vert_{L^{2}(\Omega)}^{2},
	\end{align*}
	one has for any $t\in (0,T]$
	\begin{equation}\label{esimate of regular by integrating}
		\begin{aligned}
			\frac{1}{4} \Vert \mathcal{A}_\epsilon  \m^\epsilon (t) \Vert_{L^2(\Omega)}^2 
			\le & C +
			C \int_0^t \Vert \mathcal{A}_\epsilon  \m^\epsilon (\tau)\Vert_{L^2(\Omega)}^6 \d \tau,
		\end{aligned}
	\end{equation}
	where $C$ depends on $\Vert\mathcal{A}_\epsilon \m^\epsilon_{\mathrm{init}}\Vert_{L^{2}(\Omega)}$, $\G^\epsilon[\m^\epsilon_{\mathrm{init}}]$ thus is independent of $\epsilon$ and $t$ by assumption \eqref{compatibility condition of initial data}-\eqref{initial data} and Lemma \ref{lemma: regularity for epslon}.
	Denote the right-hand side of \eqref{esimate of regular by integrating} by $F(t)$ and write
	\begin{equation*}
		\frac{\d}{\d t} F(t) \le C F^3(t).
	\end{equation*}
	By the Cauchy-Lipshitz-Picard Theorem \cite{brezis2011functional} and comparison principle, there exists $T^*\in (0,T]$ independent of $\epsilon$, such that $F(t)$ is uniformly bounded on $[0,T^*]$,
	thus $\Vert \mathcal{A}_\epsilon  \m^\epsilon (t) \Vert_{L^2(\Omega)}^2 $ is uniformly bounded by \eqref{esimate of regular by integrating}.  The Lemma is proved.
\end{proof}

\section*{Acknowledgments}
J. Chen was supported by National Natural Science Foundation of China via grant 11971021. J.-G. Liu was
supported by Natural Science Foundation via grant DMS-2106988.
Z. Sun was supported by the Postgraduate Research \& Practice Innovation Program of Jiangsu
Province via grant KYCX21\_2934.

\bibliographystyle{amsplain}
\bibliography{ref}

\end{document}